\documentclass[pdflatex,sn-mathphys,Numbered,iicol]{sn-jnl}% Math and Physical Sciences Reference Style

\usepackage{bm}
\usepackage{aliascnt}
\usepackage[usenames,dvipsnames]{color}
\usepackage{graphicx}%
\usepackage{multirow}%
\usepackage{amsmath,amssymb,amsfonts}%
\usepackage{amsthm}%
\usepackage{mathrsfs}%
\usepackage[title]{appendix}%
\usepackage{xcolor}%
\usepackage{textcomp}%
\usepackage{manyfoot}%
\usepackage{booktabs}%
\usepackage{algorithm}%
\usepackage{algorithmicx}%
\usepackage{algpseudocode}%
\usepackage{listings}%
\usepackage{bbm}
\usepackage[T1]{fontenc}
\usepackage{tikz}
\usetikzlibrary{shapes.geometric, arrows, positioning}

\newtheorem{proposition}{Proposition}%

\newtheorem{example}{Example}%
\newtheorem{remark}{Remark}%

\newcommand{\N}{\mathbb{N}}

\newcommand{\R}{\mathbb{R}}
\newcommand{\Rplus}{\mathbb{R}_+}

\let\RE\Re
\let\Re=\undefined
\DeclareMathOperator{\Re}{\RE e}
\let\IM\Im
\let\Im=\undefined
\DeclareMathOperator{\Im}{\IM m}
\DeclareMathOperator*{\argmin}{argmin}
\newcommand{\abs}[1]{\left|#1\right|}
\newcommand{\norm}[1]{\left\|#1\right\|}
\newcommand{\set}[1]{\left\{#1\right\}}
\newcommand{\Set}[2]{\left\{#1 ~ : ~ #2\right\}}
\newcommand{\inner}[2]{\left<#1,#2\right>}
\renewcommand{\d}{\,\mathrm d}
\newcommand{\PP}{{\mathbb P}}

\newcommand{\dom}{D}
\newcommand{\ulow}{u^\mathrm{low}}
\newcommand{\Ulow}{{\bm u}^\mathrm{low}}
\newcommand{\uupp}{u^\mathrm{upp}}
\newcommand{\Uupp}{{\bm u}^\mathrm{upp}}
\newcommand{\normlap}{\tilde{\Delta}}
\newcommand{\Normlap}{\tilde{\bm \Delta}}
\newcommand{\normnabla}{\tilde{\nabla}}
\newcommand{\Normnabla}{\tilde{\bm \nabla}}
\newcommand{\normpartial}{\tilde{\partial}}
\renewcommand{\u}{{\bm u}}
\newcommand{\x}{{\bm x}}
\newcommand{\y}{{\bm y}}

\renewcommand{\v}{{\bm v}}
\newcommand{\w}{{\bm w}}
\newcommand{\f}{{\bm f}}
\newcommand{\F}{{\bm F}}
\newcommand{\Y}{{\bm Y}}
\newcommand{\U}{{\bm U}}
\newcommand{\TV}{\mathbf{TV}}
\newcommand{\normtv}{\tilde{\mathrm{TV}}}
\newcommand{\Normtv}{\tilde{\mathbf{TV}}}
\newcommand{\Blanket}{\bar{\bm u}}
\newcommand{\blanket}{\bar u}
\newcommand{\Flow}{{\bm f}^\mathrm{low}}
\newcommand{\Fupp}{{\bm f}^\mathrm{upp}}
\newcommand{\Idmat}{{\bm I}}
\newcommand{\ddiff}{{\bm A}}
\newcommand{\ddiffflat}{\tilde{\bm A}}
\newcommand{\UU}{\mathbb{U}}
\newcommand{\VV}{\mathbb{V}}
\newcommand{\given}{{~ | ~}}
\newcommand{\flatten}{\mathrm{flat}} 
\newcommand{\fwd}{\mathcal{G}}
\newcommand{\Fwd}{\bm G}
\newcommand{\indicator}{\mathbbm{1}}
\newcommand{\normal}{\mathcal N}
\newcommand{\Ulows}{\bm u^{\mathrm{low}, s}}
\newcommand{\Uupps}{\bm u^{\mathrm{upp}, s}}

\newcommand{\zlow}{\bm z^\mathrm{low}}
\newcommand{\zupp}{\bm z^\mathrm{upp}}

\begin{document}

\title[Uncertainty Quantification for Scale-Space Blob Detection]{Uncertainty Quantification for Scale-Space Blob Detection}

% HACK FOR CORRECT AFFILIATIONS IN IICOL FORMAT 
\author*[1]{\fnm{Fabian} \sur{Parzer}}\email{fabian.kai.parzer@univie.ac.at}

\author[1]{\fnm{Clemens} \sur{Kirisits}}
\email{clemens.kirisits@univie.ac.at}

\author[1,2,3]{\fnm{Otmar} \sur{Scherzer}}\email{otmar.scherzer@univie.ac.at}

\affil[1]{\orgdiv{Faculty of Mathematics}, \orgname{University of Vienna}, \orgaddress{\street{Oskar-Morgenstern-Platz 1}, \city{Vienna}, \postcode{1090}, \country{Austria}}}

\affil[2]{\orgname{Johann Radon Institute for Computational and Applied Mathematics (RICAM)}, \orgaddress{\street{Altenbergerstrasse 69}, \city{Linz}, \postcode{4040}, \country{Austria}}}

\affil[3]{\orgname{Christian Doppler Laboratory for Mathematical Modeling and Simulation of Next Generations of Ultrasound Devices (MaMSi)}, \orgaddress{\street{Oskar-Morgenstern-Platz 1}, \city{Vienna}, \postcode{1090}, \country{Austria}}}

\abstract{We consider the problem of blob detection for uncertain images, such as images that have to be inferred from noisy measurements. Extending recent work motivated by astronomical applications, we propose an approach that represents the uncertainty in the position and size of a blob by a region in a three-dimensional scale space. Motivated by classic tube methods such as the taut-string algorithm, these regions are obtained from level sets of the minimizer of a total variation functional within a high-dimensional tube. The resulting non-smooth optimization problem is challenging to solve, and we compare various numerical approaches for its solution and relate them to the literature on constrained total variation denoising. Finally, the proposed methodology is illustrated on numerical experiments for deconvolution and models related to astrophysics, where it is demonstrated that it allows to represent the uncertainty in the detected blobs in a precise and physically interpretable way.}


\maketitle

\section{Introduction}\label{sec:intro}

Blob detection is a classic task in computer vision. Here, we mean by a \emph{blob} a round structure with a roughly Gaussian intensity profile. In order to simultaneously estimate the position and size of such blobs, detection methods often rely on the scale-space representation of an image, which represents the image at different levels of smoothing, allowing to distinguish low-scale from high-scale structures. This approach is commonly refered to as scale-space blob detection. The most well-known example of this is the Laplacian-of-Gaussians (LoG) method \cite{Lin98}, which is based on the premise that, for the Gaussian scale-space representation, the local extrema of the normalized Laplacian are good indicators for the position and size of a blob.

In many applications -- in particular in astronomy -- the image of interest is not known a-priori but has to be reconstructed from noisy measurements. This means that the image comes with significant uncertainties, and it is important to take these uncertainties into account when performing blob detection.

A particular example from astronomy is integrated-light stellar population recovery \cite{BoeAlfNeuMarLea20}, where the problem is to detect stellar populations as blobs in a two-dimensional mass density that is reconstructed from an observed spectrum. For this problem, the present authors with co-authors have developed an uncertainty-aware version of the Laplacian-of-Gaussians blob detection method, ULoG \cite{ParJetBoeAlfSch23a}. ULoG was formulated as a tube method that detects significant blobs by computing a tube-shaped confidence region for the uncertain scale-space representation, and then solving a minimization problem designed to obtain a representative that exhibits the "least amount of blobs".

While the results of the ULoG method were satisfying for this particular application, it only yielded a very rudimentary representation of the uncertainty with respect to the position and size of a blob.

In this paper, we propose an improved method that aims to resolve this issue. The basic idea is to represent the uncertainty in a blob by a region in scale space which represents the possible variation in position and size.

To obtain these regions, we formulate an optimization problem that enforces solutions with piecewise constant normalized Laplacian, from which the desired blob regions can be extracted as level sets. The formulation of the optimization problem uses ideas from total variation (TV) regularization, which is why we refer to the novel method as TV-ULoG.

\subsection{Contributions}

\begin{itemize}
\item We introduce the TV-ULoG method for blob detection with uncertainties. The proposed approach is flexible and can be adapted to a wide range of applications, in particular Bayesian imaging. We also discuss connections to the taut-string algorithm and constrained total-variation denoising.
\item We extensively discuss the numerical treatment of the resulting non-smooth, bound-constrained convex optimization problem. We compare approaches based on smoothing the dual or the primal problem, and an interior-point approach based on reformulating the optimization problem as SOCP.
\item Finally, we illustrate the TV-ULoG method on numerical test cases for astronomical imaging and deconvolution.
\end{itemize}

\subsection{Organization of the paper}

The paper is organized as follows:
\begin{itemize}
\item In the remainder of this section, we review related work (\autoref{subsec:related_work}) and introduce notation that is used throughout the paper (\autoref{subsec:notation}).
\item In \autoref{sec:preliminaries}, we recall the necessary prerequisites on scale-space blob detection. We focus on the Gaussian scale-space representation and the Laplacian-of-Gaussians blob detection method, which are fundamental for the rest of the paper.
\item In \autoref{sec:tube_methods}, we describe our tube-based approach for uncertainty-aware blob detection. After discussing scale-space aspects of uncertainty quantification and the ULoG method from our previous work, we derive the novel TV-ULoG method.
\item In \autoref{sec:numerical_treatment}, we discuss in detail the numerical implementation of TV-ULoG. The majority of the section is devoted to the solution of the resulting optimization problem.
\item In \autoref{sec:numerics}, we demonstrate the method on two numerical test cases. We also use these test cases to evaluate the performance of the proposed optimization methods.
\item The paper ends with a conclusion in \autoref{sec:conclusion}.
\end{itemize}

\subsection{Related work}\label{subsec:related_work}

We have based our work on the Laplacian-of-Gaussians method for scale-space blob detection \cite{Lin98} since it is well-known and its mathematical formulation is simple, making it easier to extend it to the case of uncertain images. Alternative methods for blob detection are discussed e.g. in the reviews \cite{MikTuySchZisMat05, Lin13, LiWanTiaDin15}. Some general references on scale-space methods for image processing and computer vision are \cite{Lin94, SpoNieFloJoh97, Lin94a, Wei98}.

Our work can be seen as an instance of a statistical scale space method \cite{HolPas17}, but this is not a focus of this paper. In particular the works \cite{GodMarCha04, GodOig05, ThoRueSkrGod12, PasLauHol13} are similar since they also study uncertain signals in scale space. However, our approach differs through its formulation as a tube method and the specific focus on blob detection.

Another related line of work formulates significance tests for image structures as convex optimization problems \cite{RepPerWia19, PriMcECaiKitWal21, PriMcEPraKit20}. This methodology relies on concentration inequalities \cite{Per17} and is computationally very efficient, but does not automatically detect the position and scale of structures since the presence of a structure must be formulated as user-specified hypothesis.

To our knowledge, before \cite{ParJetBoeAlfSch23a}, the specific problem of uncertainty-aware blob detection has not been addressed previously.

\subsection{Notation}\label{subsec:notation}

\begin{itemize}
\item For $n \in \N$, we define the discrete range $[n] := \set{1,\ldots, n}$.
\item $\Rplus := [0, \infty)$ denotes the non-negative real numbers
\item For a vector $\x \in \R^n$, its Euclidean norm is denoted by $\norm{\x} := \sqrt{\sum_{i=1}^n x_i^2}$.
\item The convolution of two functions is denoted by $f * g(\x) := \int f(\y) g(\x - \y) \d \y$.
\item The spatial Laplace operator is denoted by $\Delta := \partial_{x_1}^2 + \partial_{x_2}^2$.
\item The probability distribution of a random element $U$ is denoted by $\PP_U$ \cite{Kal02}. Its corresponding density function is denoted by $p_U$, if it exists. Given another random element $V$, the conditional distribution of $U$, given $V=v$, is denoted by $\PP_{U | V}(\cdot | v)$, with corresponding conditional density $p_{U| V}(\cdot | v)$ (see also \cite{Sch95}).
\item Given functions $\ulow,\uupp: D \to \R$ on a domain $D \subset \R^n$, with $\ulow \leq \uupp$, we call the set of functions
\begin{equation}
\begin{aligned}
[u^\mathrm{low},u^\mathrm{upp}] := \lbrace & u: D \to \R ~ : \\
 & u^\mathrm{low}(\x) \leq u(\x) \leq u^\mathrm{upp}(\x) ~ \forall \x \in D \rbrace
\end{aligned} \label{eq:tube}
\end{equation}
a \emph{tube}. Similarly, given two vectors $\u^\mathrm{low},\u^\mathrm{upp} \in \R^n$, we call the set of vectors
\begin{equation}
\begin{aligned}
[\u^\mathrm{low},\u^\text{upp}] := \lbrace \u \in \R^n ~: ~\ulow_i \leq u_i \leq \uupp_i  \\
\forall i \in [n] \rbrace
\end{aligned}\label{eq:discrete_tube}
\end{equation}
a \emph{discrete tube}. This definition is straightforward to extend to higher-dimensional objects, such as discrete images.
\item We denote the characteristic function of a set $C$ by
\begin{align*}
\chi_C(\v) := \begin{cases} 0, & \text{if } \v \in C, \\
\infty, & \text{otherwise}. \end{cases}
\end{align*}
\item We let $\bm{0}_n \in \R^n$ denote the zero-vector in $\R^n$ and $\bm{1}_n \in \R^n$ denote the vector with all entries equal to 1. Similarly, $\bm{0}_{m \times n} \in \R^{m \times n}$ denotes the zero matrix and $\bm{1}_{m \times n} \in \R^{m \times n}$ denotes the matrix with all entries equal to 1. Also, $\bm e_i^n \in \R^n$ denotes the $i$-th basis vector in $\R^n$, with entries $(\bm e_i^n)_j= \delta_{ij}$.
\item Given a nonempty closed convex set $C \subset \R^d$, $P_C(x) := \argmin_{\y \in C} \norm{\y - \x}$ denotes the projection on $C$.
\item For a set $S$, $2^S$ denotes its power set.
\end{itemize}

\section{Scale-space blob detection}\label{sec:preliminaries}\label{sec:scale_space}

In this section, we review the scale-space approach to blob detection that underlies the rest of this paper. We focus on the Gaussian scale-space representation, which we introduce in \autoref{subsec:scale_space_representation}. Then, we review the classic Laplacian-of-Gaussians method for blob detection in \autoref{subsec:blob_detection}.

\subsection{Scale-space representations}\label{subsec:scale_space_representation}

In the computer vision literature, the mathematical theory of describing images at different scales is known as \emph{scale space theory} \cite{Lin94}. A scale space representation of an image $f: \R^2 \to \R$ is a function $u: \R^2 \times \Rplus \to \R$ which depends on an additional third parameter $t$ that represents physical scale. The \emph{Gaussian scale-space representation} is the most studied example, due to its simple formulation and the fact that it is the unique linear scale space representation that satisfies a series of intuitive axioms that formalize the notion of scale \cite{AlvGuiLioMor93, FloTerKoeVie92}.

The Gaussian scale-space representation of an image $f: \dom \to \R$ on a domain $D \subset \R^2$ is defined as the solution $u: D \times [0, \infty) \to \R$ of the diffusion equation
\begin{equation}
\begin{aligned}
\partial_t u(\x, t) &= \frac{1}{2} \Delta u(\x, t), \quad &(\x, t) \in D \times (0, \infty),\\
u(\x, 0) &= f(\x), \quad &\x \in D,\\
\partial_{\bm \nu} u(\x, t) &= 0, \quad &(\x, t) \in \partial D \times (0, \infty).
\end{aligned}\label{eq:continuous_scale_space}
\end{equation}
Here, we imposed Neumann boundary conditions, following \cite{DuiFelFloPla03}. In the following, we will denote with $\Phi$ the solution operator of \eqref{eq:continuous_scale_space} that maps an image $f$ to its scale-space representation $u$. One can show that $\Phi$ is well-defined under suitable assumptions (see e.g. \cite{Bre11}).

For the rest of this paper, we will not consider other scale-space representations. In particular, we will always mean the representation $u$ defined in \eqref{eq:continuous_scale_space} when we write "the scale-space representation" of an image $f$.

\subsection{Blob detection}\label{subsec:blob_detection}

The scale-space representation of an image is often used to detect the position and size of features of interest. A particular example of this is the well-known Laplacian-of-Gaussians method for blob detection \cite{Lin98}. It is a special case of the differential-invariants approach for feature detection which we introduce next.

\subsubsection{Feature detection from differential invariants}\label{subsubsec:scale_invariance}

Image features often correspond well to local extrema of \emph{scale-normalized derivatives} of the image's scale-space representation $u$, i.e. combinations of the scaled partial derivatives
\begin{align}
\normpartial_{x_i} u(\x, t) := \sqrt{t} \partial_{x_i} u(\x, t), \quad i=1,2 \label{eq:scale_normalized_derivative}
\end{align}
(cf. \cite{Lin13}). The scale-normalization is necessary to achieve an intuitive scale-invariance property: A feature of scale $t$ in the image $f$ should correspond to a feature of scale $s \cdot t$ in the rescaled image $f_s(\x) := f(\x / \sqrt{s})$. E.g. zooming in by a factor 2 increases the scale of all features by a factor 4.

\subsubsection{The Laplacians-of-Gaussians method}

The Laplacian-of-Gaussians method is a special case of the methodology described in \autoref{subsubsec:scale_invariance} for blob detection. It uses the local minimizers of the scale-normalized Laplacian of $u$ as indicators for blobs in an image, where the scale-normalized Laplacian is given by
\begin{align}
\normlap u(\x, t) := (\normpartial_{x_1}^2 + \normpartial_{x_2}^2) u(\x, t) = t \Delta u(\x, t), \label{eq:normalized_laplacian}
\end{align}
That is, a local minimizer or maximizer of $\normlap u$ at $(\x, t) \in \dom \times \Rplus$ indicates the presence of a blob-like shape with center $\x$ and scale $t$.

\begin{example}\label{ex:prototypical_blob}
To understand what is meant by "blob-like shape", let us define the prototypical blob with center $\bm m \in \R^2$ and size $s$ as the symmetrical Gaussian
\begin{align*}
f(\x) := \frac{1}{2 \pi s} \exp \left( - \frac{\norm{\x - \bm m}^2}{2 s} \right), \quad \x \in \R^2.
\end{align*}
Its scale-space representation is
\begin{align*}
u(\x, t) = \frac{1}{2 \pi (s + t)} \exp \left( - \frac{\norm{\x - \bm m}^2}{2 (s + t)} \right),
\end{align*}
and the scale-normalized Laplacian is
\begin{align*}
\normlap u(\x, t) = t \frac{\norm{\x - \bm m}^2 - 2(s + t)}{2 (s + t)^2} u(\x, t),
\end{align*}
which has a unique local minimum at $(\bm m, s)$. This means that the position and size of the prototypical Gaussian blob $f$ are exactly recovered from the local minimizer of $\normlap u$. Note that the normalization in \eqref{eq:normalized_laplacian} is important for detecting the scale, since the un-normalized Laplacian $\Delta u$ has a unique local minimum at $(\bm m, 0)$, which does not indicate the blob size. The relation of the Laplace operator to blob detection is also discussed in \cite{MarHil80, VooPog88}.
\end{example}

The prototypical Gaussian blob also motivates the common visualization of the results of the Laplacian-of-Gaussians method, where a detected blob $(\bm x, t)$ is  visualized by a circle with center $\x$ and radius proportional to $\sqrt{t}$. Usually, the radius $\sqrt{t}$ is used for one-dimensional signals and the radius $\sqrt{2t}$ is used for two-dimensional signals (images).

\section{Tube-based uncertainty quantification for blob detection}\label{sec:tube_methods}

The Laplacian-of-Gaussians method described in \autoref{subsec:blob_detection} detects blobs from local minimizers of $\normlap u$, where $u$ is the Gaussian scale-space representation of $f$ given by \eqref{eq:continuous_scale_space}. The purpose of this paper is to extend this methodology to the case where the image $f$ is uncertain, for example when it has to be estimated from noisy measurements.

\subsection{Incorporating uncertainty into scale-space methods}\label{subsec:scale_space_uq}

Consider the problem of recovering an image of interest $f^*: D \to \R$ (called the \emph{ground truth}) from a noisy measurement $y$ given by
\begin{align}
y = \fwd(f^*) + w, \label{eq:imaging_inverse_problem}
\end{align}
where $\fwd$ is an operator that represents the measurement process, and $w$ is noise. The presence of the noise implies that any estimate of the image $f^*$ comes with uncertainties. In the Bayesian approach \cite{GemGem86, Han93} to imaging, these uncertainties are taken into account by modelling $f^*$, $y$ and $w$ as realizations of random quantities $F$, $Y$ and $W$ that are related by
\begin{align}
Y = \fwd(F) + W. \label{eq:forward_equation}
\end{align}
The assumed distribution $\PP_F$ of the random image $F$ is called the \emph{prior}, since it encodes a-priori assumptions on the unknown image. Using \eqref{eq:forward_equation} and statistical assumptions on $W$, one can define a likelihood in the form of a conditional probability distribution $\PP_{Y|F}$. Recall that, for given $f$, $\PP_{Y|F}(\cdot|f)$ is a probability measure that represents the distribution of $Y$ given $F=f$. For the construction of conditional probability distributions on abstract spaces (e.g. infinite-dimensional function spaces), see for example \cite[chapter 6]{Kal02} or \cite[chapter 8.3]{Kle14}.

Together, the prior and the likelihood determine the so-called \emph{posterior} distribution $\PP_{F | Y}$, which quantifies the uncertainty with respect to $F$, conditional on observing $Y$, through Bayes' rule (see e.g. \cite{KaiSom05, DasStu17} for further reference).

Treating the image of interest as random means that its scale-space representation also needs to be modelled as random object. Let $\Phi$ denote the solution operator of the diffusion equation \eqref{eq:continuous_scale_space} that maps an image $f$ to its corresponding scale-space representation $u=\Phi f$. Then the scale-space representation of the random image $F$ is the random function $U = \Phi F$. The posterior distribution $\PP_{U|Y}$ of $U$ is then determined by the posterior distribution $\PP_{F|Y}$ of $F$ through the usual transformation rules: Given an observation $Y=y$, the posterior probability that $U$ lies in a set of functions $A$ is given by \\
\begin{align}
\PP_{U | Y}(A \given y) & = \PP_{\Phi F | Y}(A \given y) \notag\\
& = \PP_{F | Y}(\Phi^{-1}(A) \given y).\label{eq:pushforward_continuous}
\end{align}
The problem of uncertainty-aware blob detection can then be rephrased as finding local minimizers of $\normlap U$ for uncertain $U$. Developing a practical method to solve this problem requires a suitable representation of the uncertainty encoded in the abstract posterior distribution $\PP_{U | Y}$. In this paper, we assume that the uncertainty with respect to $U$ is represented by a \emph{credible scale-space tube}. That is, we assume knowledge of functions $\ulow, \uupp: D \times \Rplus \to \R$ such that
\begin{equation}
\PP_{U | Y}([\ulow, \uupp] \given y) \geq 1 - \alpha, \label{eq:credibility_condition}
\end{equation}
for a small credibility parameter $\alpha \in (0, 1)$. We restrict ourselves to tube-shaped regions for three main reasons: First, such tubes can be obtained in many applications (see \autoref{rem:compute_tubes}). Second, this choice leads to a relatively simple formulation of our method as bound-constrained convex optimization problem (see \autoref{subsubsec:tv_ulog_formulation} below). Finally, it is motivated by existing tube methods from density estimation, which we quickly review in \autoref{subsec:tube_methods} below.

\begin{remark}\label{rem:compute_tubes}
Since $U$ depends linearly on $F$, it is often straightforward to compute tubes that satisfy \eqref{eq:credibility_condition} using sampling-based inference (e.g. Markov chain Monte Carlo (MCMC) \cite{BroGelJonMen11}) or approaches based on analytic approximations (e.g. variational inference \cite{BleKucMcA17}). We illustrate this in \autoref{app:tubes_from_mcmc}, where we present a method for estimating the scale-space tube using MCMC samples. This method is used in the numerical examples of \autoref{sec:numerics}.
\end{remark}

\subsubsection*{The scale-uncertainty tradeoff}

The structure of the tube $[\ulow, \uupp]$ will in general represent a tradeoff between scale and uncertainty in the image $F$: At lower scales, the uncertainty is higher since small-scale features are more difficult to detect from noisy observations. At higher scales, the uncertainty decreases as this local variability is filtered out and the ability to resolve finer details is lost. This phenomenon is a special case of the fundamental bias-variance tradeoff that applies to uncertain signals in general \cite{Was06}. In the language of image processing, it corresponds to the simple intuition that coarse structures are easier to identify than fine details, given limited data. In particular, the bounds $\ulow$ and $\uupp$ in general are not scale-space representations of corresponding images $f^\mathrm{low}, f^\mathrm{upp}$. This was also demonstrated in the previous work \cite{ParJetBoeAlfSch23a}.

\subsection{Tube methods in density estimation}\label{subsec:tube_methods}

The task of estimating the local extrema of an uncertain signal has previously been studied in density estimation, where it is adressed using tube methods \cite{Mam95,Mamvan97,DavKov04,ObeSchKov07}. For an unknown density $g: \dom \to \R$, one considers a tube $[g^\mathrm{low}, g^\mathrm{upp}]$ that is typically constructed as a corridor of fixed width around a noisy measurement of the density.

The minimizer $\bar g \in [g^\mathrm{low}, g^\mathrm{upp}]$ of a tube-constrained optimization problem
\begin{align*}
\min_g \quad & \mathcal J(g) \\
\text{s. t.} \quad & g^\mathrm{low} \leq g \leq g^\mathrm{upp}
\end{align*}
then serves as representative with the "smallest amount of features" among all signals in the tube, where the "amount of features" is encoded in the choice of the cost function $\mathcal J$. A prominent example is the one-dimensional taut-string algorithm \cite{HarHar85, Dav95}. It uses the choice $\mathcal J(g) = \int \sqrt{1 + \abs{g'}^2}$ which is known to yield an estimate $\bar g$ such that its derivative $\bar g'$ is piecewise-constant and minimizes the number of local extrema amongst all functions with anti-derivative in the given tube \cite{DavKov01}.

\subsection{Review of the ULoG method}\label{subsec:ulog}

Motivated by these classic tube methods for density estimation, we define a suitable cost function $\mathcal J(u)$ that serves as a proxy for the number of local extrema of $\normlap u$, and then obtain a desired representative $\blanket$ as the minimizer of the constrained optimization problem
\begin{equation}
\begin{aligned}
\min_u \quad & \mathcal J(u) \\
\text{s. t.} \quad & \ulow \leq u \leq \uupp \text{ on } \dom \times \Rplus.
\end{aligned}\label{eq:blanket_problem}
\end{equation}
In previous work \cite{ParJetBoeAlfSch23a} we used the cost function
\begin{align}
\mathcal J(u) = \int_{\dom \times \Rplus} (\normlap u)^2. \label{eq:old_cost_function}
\end{align}
The particular difference to our new method is that we minimize a different energy function, which takes into account derivatives with respect to $t$ (see \eqref{eq:tv_ulog} below), while the optimization problem \eqref{eq:blanket_problem} can be decoupled in time, leading to a group of bounded linear least-squares problems that can be solved independently.

In analogy to the Laplacians-of-Gaussians method, we then identified "significant blobs" as the local minimizers of the solution $\blanket$ of \eqref{eq:blanket_problem} (see also \autoref{ex:prototypical_blob}). The resulting method was correspondingly called the "Uncertainty-aware Laplacian-of-Gaussians" (ULoG) method.

The ULoG method showed satisfying results in simulations, where minimizers of \eqref{eq:blanket_problem} with $\mathcal J$ given by \eqref{eq:old_cost_function} typically exhibited a scale-normalized Laplacian $\normlap \blanket$ that attained its local minima at distinct points which corresponded very well to significant blobs in the ground truth image (cf. \cite{ParJetBoeAlfSch23a}).

However, the choice \eqref{eq:old_cost_function} was motivated more by computational simplicity than theoretical considerations. Furthermore, a main limitation of this approach is that it only provides a limited view of the uncertainty of scale-space blobs, since the uncertainty with respect to a blob is represented by a single point $(\x, t) \in \dom \times \Rplus$. The scale parameter $t$ is difficult to interpret, since it is both related to the expected physical size of the blob and to the uncertainty with respect to its center. That is, a large value of $t$ can correspond both to the presence of a large, certain blob in $f$, or a small one with uncertain position.

\subsection{A total variation-based approach}\label{subsec:tv_approach}

A remedy is to instead represent the uncertainty with respect to a blob by a three-dimensional connected \emph{region} in $\dom \times \Rplus$ that contains the uncertain blob with high probability. The geometry of this region then provides a more nuanced view of the possible variation in position and size.

We determine the desired regions using again a tube-based approach. To this end, we consider a cost function that leads to minimizers of \eqref{eq:blanket_problem} with \emph{piecewise-constant} normalized Laplacian. In that case, the local minima of $\normlap \blanket$ are attained at flat connected regions which can easily be extracted using thresholding.

As cost function, we propose to use the total variation of $\normlap u$ \cite{CasChaNov15}. This can be motivated by the fact that in one dimension, minimizing the total variation of $u'$ within the tube $[\ulow, \uupp]$ leads to an estimate with very similar properties as the mode-minimizing taut-string estimate \cite{DavKovMei09}. On the other hand, this approach is consistent with insights from image processing, where total variation regularization is a common tool to achieve piecewise-constant reconstructions \cite{RudOshFat92}. 

\subsubsection{Outline of the method}

We start with a high-level description of the proposed method (see also the schematic overview in \autoref{fig:flowchart}).

\begin{enumerate}
\item We assume that we are given a prior distribution $\PP_F$ for the image of interest and a likelihood $\PP_{Y|F}$ (determined by the forward model \eqref{eq:forward_equation} and statistical assumptions on the noise). For a concrete observation $Y=y$, these determine a posterior distribution $\PP_{F|Y}(\cdot | y)$ through Bayes' rule (see \autoref{subsec:scale_space_uq}), which in turn determines a posterior distribution $\PP_{U|Y}(\cdot|y)$ for the uncertain scale-space representation $U$ (see \eqref{eq:pushforward_continuous}).
\item For a given credibility level $\alpha \in (0,1)$, we compute a credible scale-space tube $[\ulow, \uupp]$ that satisfies \eqref{eq:credibility_condition}, e.g using the MCMC-based method described in \autoref{app:tubes_from_mcmc}.
\item Next, we solve the tube-constrained total variation-based optimization problem formulated in \autoref{subsubsec:tv_ulog_formulation} below. This yields a minimizer $\blanket \in [\ulow, \uupp]$.
\item From the computed minimizer $\blanket$, we extract the regions in $D \times \Rplus$ where $\normlap \blanket$ attains its local minima. In practice this is done using the thresholding procedure described in \autoref{subsec:extraction}.
\item The extracted blob regions can now be visualized using the method outlined in \autoref{subsec:visualization}.
\end{enumerate}%
Since this approach can be seen as a modification of the previous ULoG method (\autoref{subsec:ulog}), we will refer to it as TV-ULoG.

\tikzstyle{default} = [rectangle, rounded corners, 
minimum width=1cm, 
minimum height=1cm,
text centered, 
draw=black, 
fill=red!30]

\tikzstyle{arrow} = [thick,->,>=stealth]

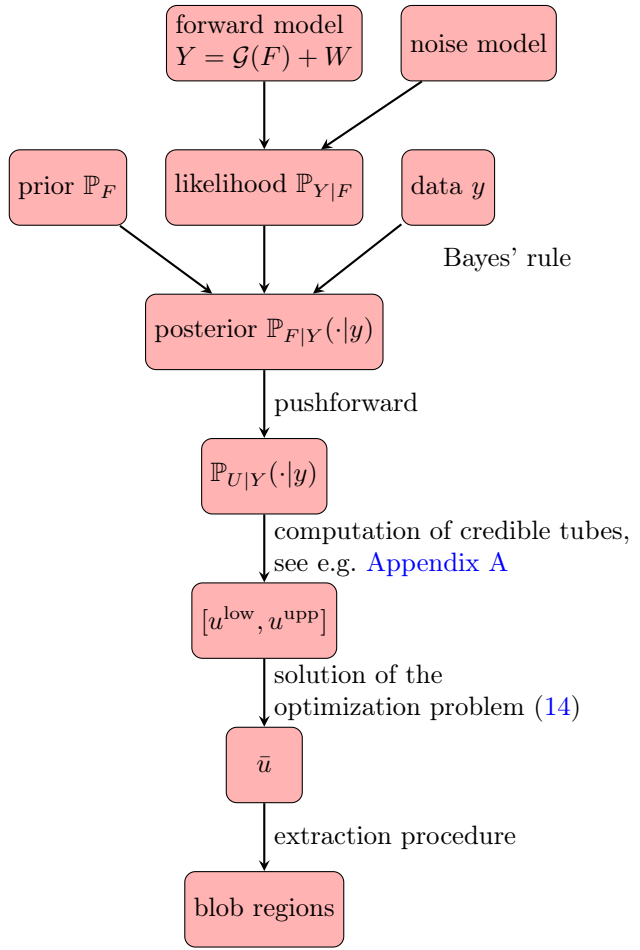
\begin{figure}[h]
\begin{tikzpicture}[node distance = 0.9cm and 0.5cm]

\node (forward) [default, align=left,] {forward model \\ $Y = \fwd(F) + W$};
\node (noise) [default, right=of forward] {noise model};
\node (likelihood) [default, below=of forward] {likelihood $\PP_{Y|F}$};
\node (prior) [default, left=of likelihood] {prior $\PP_F$};
\node (data) [default, align=left, right=of likelihood] {data $y$};
\node (posterior) [default, below=of likelihood] {posterior $\PP_{F|Y}(\cdot|y)$};
\node (uposterior) [default, below=of posterior] {$\PP_{U|Y}(\cdot | y)$};
\node (tubes) [default, below=of uposterior] {$[\ulow, \uupp]$};
\node (tvulogsolution) [default, below=of tubes] {$\bar u$};
\node (regions) [default, below=of tvulogsolution] {blob regions};

\draw [arrow] (forward) -- (likelihood);
\draw [arrow] (noise) -- (likelihood);
\draw [arrow] (prior) -- (posterior);
\draw [arrow] (likelihood) -- (posterior);
\draw [arrow] (data) -- node[anchor=west, xshift=1cm] {Bayes' rule} (posterior);
\draw [arrow] (posterior) -- node[anchor=west] {pushforward}(uposterior);
\draw [arrow] (uposterior) -- node[anchor=west, align=left] {computation of credible tubes, \\ see e.g. \autoref{app:tubes_from_mcmc}}(tubes);
\draw [arrow] (tubes) -- node[anchor=west,align=left] {solution of the \\optimization problem \eqref{eq:tv_ulog}}(tvulogsolution);
\draw [arrow] (tvulogsolution) -- node[anchor=west] {extraction procedure}(regions);
\end{tikzpicture}
\caption{Overview of the proposed approach}\label{fig:flowchart}
\end{figure}

\subsubsection{Formulation of the optimization problem}\label{subsubsec:tv_ulog_formulation}

To prepare the precise mathematical formulation of the resulting optimization problem, we define the scale-normalized total variation of a function $u: \dom \times \Rplus \to \R$ by
\begin{align}
\normtv(u) := \int_{\dom \times \Rplus} \norm{\normnabla_{\x, t} u}, \label{eq:total_variation}
\end{align}
where
\begin{align}
\normnabla_{\x, t} := \begin{bmatrix}
\normpartial_{x_1} \\ \normpartial_{x_2} \\ \normpartial_t
\end{bmatrix} \label{eq:dnorm}
\end{align}
is the scale-normalized gradient operator. Here, $\normpartial_{x_1}$ and $\normpartial_{x_2}$ are the scale-normalized spatial derivatives defined in \eqref{eq:scale_normalized_derivative}, while $\normpartial_t$ is defined as
\begin{align*}
\normpartial_t u(\x, t) := t \partial_t u(\x, t).
\end{align*}
Following the discussion at the start of this section, we suggest to use the choice $\mathcal J(u) = \normtv(\normlap u)$ in \eqref{eq:blanket_problem}. To allow for a finite-difference discretization of the involved differential operators, we have to assume suitable boundary conditions. Our particular choice of Neumann boundary conditions on $u$ and $\normlap u$ was mostly motivated by ease of implementation (see \autoref{rem:boundary_conditions}). In summary, we arrive at the following formulation:
\begin{equation}
\begin{aligned}
\min_{u} \quad & \normtv(\normlap u) \\
\text{s. t.} \quad & \ulow \leq u \leq \uupp \text{ on } \dom \times \Rplus, \\
 & \partial_{\bm \nu} u(\bm x, t) = 0 \text{ on } \partial \dom \times (0, \infty), \\
 & \partial_{\bm \nu} \normlap u(\bm x, t) = 0 \text{ on } \partial \dom \times (0, \infty).
\end{aligned}\label{eq:tv_ulog}
\end{equation}
In \autoref{subsec:implementation_tv}, we discuss the discretization of \eqref{eq:tv_ulog} and consider various optimization algorithms for finding a minimizer of the resulting non-smooth convex optimization problem. Extraction and visualization of the desired regions from the computed minimizer are described in \autoref{subsec:extraction} and \autoref{subsec:visualization}.

\begin{remark}
The scale-normalization in \eqref{eq:dnorm} is necessary to achieve a scale-invariance property analogous to the one described in \autoref{subsubsec:scale_invariance}: If $\bar u_s$ is a minimizer of the scaled problem
\begin{align*}
\min_{u_s} \quad & \normtv(\normlap u_s) \\
\text{s. t.} \quad & u_s^\mathrm{low} \leq u_s \leq u_s^\mathrm{upp} \text{ on } \dom \times \Rplus, \\
 & \partial_{\bm \nu} u_s(\bm x, t) = 0 \text{ on } \partial \dom \times \Rplus, \\
 & \partial_{\bm \nu} \normlap u_s(\bm x, t) = 0 \text{ on } \partial \dom \times \Rplus,
\end{align*}
where
\begin{align*}
u_s^\mathrm{low}(\x, t) = \ulow(\sqrt{s}\x, st), \\
u_s^\mathrm{upp}(\x, t) = \uupp(\sqrt{s}\x, st),
\end{align*}
then $\bar u$ given by
\begin{align*}
\bar u(\x, t) = \bar u_s(\x/\sqrt{s}, t/s) \qquad \text{for all } (\x, t) \in \dom \times \Rplus
\end{align*}
is a minimizer of the original problem \eqref{eq:tv_ulog} and vice versa. See also \cite{Lin98} for more motivation behind such scaling properties.
\end{remark}

\begin{remark}\label{rem:boundary_conditions}
The Neumann boundary condition on $u$ in \eqref{eq:tv_ulog} is motivated by the definition of the scale-space representation \eqref{eq:continuous_scale_space}. The motivation for the second Neumann boundary condition on $\normlap u$ is that it leads to the usual formula for the discrete total variation when combined with a forward-difference approximation (see \autoref{subsec:implementation_tv}).
\end{remark}

\section{Numerical implementation}\label{sec:numerical_treatment}

In this section we discuss how the TV-ULoG method is implemented in practice. In \autoref{subsec:discrete_scale_space}, we discuss the discretization of the optimization problem \eqref{eq:tv_ulog}. Then, we present methods to solve the resulting discrete optimization problem in \autoref{subsec:implementation_tv}. \autoref{subsec:extraction} describes how the desired blob regions can be extracted from the computed minimizer, while \autoref{subsec:visualization} outlines possible visualization of these regions.

\subsection{Discrete scale-space representation}\label{subsec:discrete_scale_space}

To discuss the numerical implementation of the TV-ULoG method, we assume that the image domain $\dom \subset \R^2$ is rectangular and has been discretized into a grid $(\bm x_{ij})_{i \in [N_1], j \in [N_2]}$, with uniform grid size $h_1 > 0$ in $x_1$-direction and uniform grid size $h_2 > 0$ in $x_2$-direction, such that the discrete image of interest is given by a matrix $\f \in \R^{N_1 \times N_2}$.

For the scale discretization, it was suggested in \cite{Lin98} to use discrete scales $0 < t_1 < \ldots < t_K$ that are exponentially increasing, i.e.
\begin{align}
t_{k+1} = b t_k,  \quad k \in [K - 1] \label{eq:exponential_scales}
\end{align}
for some $b > 1$.

The discrete scale-space representation is then defined through a suitable discretization $\bm \Phi: \R^{N_1 \times N_2} \to \R^{N_1 \times N_2 \times K}$ of the solution operator $\Phi$ of the diffusion equation \eqref{eq:continuous_scale_space}. Then, the discrete scale-space representation of $\f$ is given by the three-dimensional array
\begin{align}
\bm u := \bm \Phi \f \in \R^{N_1 \times N_2 \times K}. \label{eq:discrete_scale_space_rep}
\end{align}
In practice, the solution operator $\bm \Phi$ is often implemented as a convolution with a suitable discrete convolution kernel (see e.g. \cite{Lin94a}).

Following the discussion in \autoref{subsec:scale_space_uq}, we then assume access to a credible scale-space tube $[\Ulow, \Uupp]$ for the discrete scale-space representation $\bm U = \bm{\Phi} \F$, such that, analogously to \eqref{eq:credibility_condition}, there holds
\begin{align*}
\PP_{\U | \Y}([\Ulow, \Uupp] \given \y) \geq 1 - \alpha,
\end{align*}
for given $\alpha \in (0, 1)$. How such a tube can be estimated in practice from MCMC samples is described in \autoref{app:tubes_from_mcmc}.

In \autoref{subsec:implementation_tv}, we present how, given $[\Ulow,  \Uupp]$, the non-smooth optimization problem \eqref{eq:tv_ulog} can be solved numerically after discretization. Then, in \autoref{subsec:extraction}, we describe a procedure that extracts the desired blob regions in scale space from the computed minimizer in a way that is robust against numerical errors. In \autoref{subsec:visualization}, we discuss visualizations of these regions that meaningfully represent the uncertainty of the scale-space blobs.

\subsection{Numerical treatment of the optimization problem}\label{subsec:implementation_tv}

In order to solve the optimization problem \eqref{eq:tv_ulog} numerically, we have to define suitable discretizations for the differential operators in the objective.

To discretize the scale-normalized Laplacian \eqref{eq:normalized_laplacian}, we use the common central-difference scheme, i.e. we define
\begin{equation}
\begin{aligned}
(\Normlap \bm u)_{i,j, k} := t_k &  \left( \frac{u_{i+1,j,k} - 2 u_{i,j,k} + u_{i-1,j,k}}{h_1} \right. \\
& \left. + \frac{u_{i,j+1,k} - 2 u_{i,j,k} + u_{i,j-1,k}}{h_2} \right),
\end{aligned}
\label{eq:scale_normalized_laplacian_discrete}
\end{equation}
where we mirror $u$ at the boundaries of the index range, i.e. we set
\begin{equation}
\begin{aligned}
& u_{0, j, k} := u_{2, j, k}, \quad u_{N_1+1,j,k} := u_{N_1 - 1,j,k}, \\
&  u_{i, 0, k} := u_{j,2,k}, \quad u_{i, N_2 + 1, k} := u_{j, N_2 - 1, k}
\end{aligned}\label{eq:mirror_boundary}
\end{equation}
for all $(i,j,k) \in [N_1] \times [N_2] \times [K]$. This choice implements the Neumann boundary condition in \eqref{eq:continuous_scale_space} (see example 5.49 in \cite{BreLor18}).

For the total variation \eqref{eq:total_variation}, we use an isotropic discretization. To this end, we define the forward differences in $x_1$- and $x_2$-direction by
\begin{align*}
& (\Normnabla_{\x, t} \bm u)_{i,j,k} := \begin{bmatrix} (\Normnabla_{x_1} \bm u)_{i,j,k} \\ (\Normnabla_{x_2} \bm u)_{i,j,k} \\(\Normnabla_t \bm u)_{i,j,k} \end{bmatrix} \in \R^3, \\
\text{where } & (\Normnabla_{x_1} \bm u)_{i,j,k} := t_k^{1/2}\frac{u_{i+1,j,k} - u_{i,j,k}}{h_1}, \\
& (\bm \Normnabla_{x_2} \bm u)_{i,j,k} := t_k^{1/2}\frac{u_{i,j+1,k} - u_{i,j,k}}{h_2},
\end{align*}
where we formally define
\begin{align*}
& u_{N_1 + 1,j,k} := u_{N_1,j,k}, \quad u_{i,N_2 + 1,k} := u_{i,N_2,k}, \\
& u_{i,j,K + 1} := u_{i,j,K},
\end{align*}
for all $(i,j,k) \in [N_1] \times [N_2] \times [K]$. Similar to \eqref{eq:mirror_boundary}, this choice implements the Neumann boundary condition on $\normlap u$.

For the non-uniform scale grid \eqref{eq:exponential_scales}, the forward difference approximation of the scale derivative $\partial_t u(\x, t)$ is
\begin{align*}
(\Normnabla_t \bm u)_{i,j,k} & = t_k \frac{u_{i,j,k+1} - u_{i,j,k}}{t_{k+1} - t_k} \\
& = \frac{1}{b-1} (u_{i,j,k+1} - u_{i,j,k}).
\end{align*}
We combine the forward difference approximations in a discrete scale-space gradient $\Normnabla_{\x, t}: \R^{N_1 \times N_2 \times K} \to \R^{N_1 \times N_2 \times K \times 3}$, given by
\begin{equation}
(\Normnabla_{\x, t} \bm u)_{i,j,k} := \begin{bmatrix} (\Normnabla_{x_1} \bm u)_{i,j,k} \\ (\Normnabla_{x_2} \bm u)_{i,j,k} \\(\Normnabla_t \bm u)_{i,j,k} \end{bmatrix} \in \R^3. \label{eq:discrete_forward_difference}
\end{equation}
With definition \eqref{eq:discrete_forward_difference}, we define the isotropic scale-normalized total variation of $\bm u \in \R^{N_1 \times N_2 \times K}$ as
\begin{align}
\Normtv(\bm u) = \sum_{k=1}^K \sum_{i=1}^{N_1} \sum_{j=1}^{N_2} \norm{(\Normnabla_{\x, t} \bm u)_{i,j,k}}. \label{eq:discrete_tv}
\end{align}
In summary, our discretization of \eqref{eq:tv_ulog} reads as
\begin{equation}
\begin{aligned}
\min_{\bm u \in \R^{N_1 \times N_2 \times K}} \quad & \Normtv(\Normlap \bm u) \\
\text{s.t.} \quad &\Ulow \leq \bm u \leq \Uupp.
\end{aligned}\label{eq:tv_ulog_discrete}
\end{equation}
This is a bound-constrained, non-smooth convex optimization problem. For ease of reference, we will also refer to \eqref{eq:tv_ulog_discrete} as the \emph{TV-ULoG optimization problem}.

To our knowledge, this particular problem has not been previously considered in the image processing literature. The closest candidate is the constrained total variation (CTV) denoising problem: Given a noisy image $\bm f^\delta \in \R^{N_1 \times N_2}$ and bounds $\Flow, \Fupp \in \R^{N_1 \times N_2}$ with $\Flow \leq \Fupp$, the discrete CTV-denoising problem is
\begin{equation}
\begin{aligned}
\min_{\bm f \in \R^{N_1 \times N_2}} \quad & \norm{\bm f - \bm f^\delta}^2 + \lambda \TV(\bm f) \\
\text{s. t.} \quad & \Flow \leq \bm f \leq \Fupp.
\end{aligned}\label{eq:ctv_denoising}
\end{equation}
Here, $\lambda > 0$ is a tunable regularization parameter and $\TV(\cdot)$ is the discrete total variation. The problem \eqref{eq:ctv_denoising} has been considered e.g. by Beck and Teboulle in their seminal work \cite{BecTeb09}, where they propose the fast gradient projection (FGP) method for its numerical solution. The FGP method can be summarized as applying Nesterov-accelerated projected gradient descent to the dual of \eqref{eq:ctv_denoising}. In \autoref{sssec:dual_smoothing}, we show that the FGP method can be applied to a Nesterov-smoothed dual of \eqref{eq:tv_ulog_discrete}. In \autoref{sssec:primal_smoothing}, we describe an analogous method that instead applies FGP to the Nesterov-smoothed primal problem.

However, as is illustrated further in \autoref{sec:numerics}, both methods converge very slowly for our particular problem. For this reason, we present a third alternative approach for the solution of \eqref{eq:tv_ulog_discrete} in \autoref{sssec:socp}. It formulates \eqref{eq:tv_ulog_discrete} as an equivalent second-order cone program (SOCP) and solves it with an interior-point method. This method is also motivated by analogous approaches for TV-minimization \cite{GolYin05,YinGolOsh07}.

\subsubsection{Dual smoothing}\label{sssec:dual_smoothing}

To prepare the discussion of the dual-smoothing approach, we recall the notion of Fenchel duality \cite{BotGraWan09}: Let $\UU, \VV$ be finite-dimensional Euclidean spaces, $\bm A: \UU \to \VV$ a linear function, and $\phi: \UU \to [-\infty, \infty]$, $\psi: \VV \to [-\infty, \infty]$ be proper convex functions. Then the Fenchel dual of the optimization problem
\begin{align}
\min_{\u \in \UU} \left\lbrace \phi(\u) + \psi(\bm A \u) \right\rbrace \label{eq:fenchel_primal}
\end{align}
is
\begin{align}
\max_{\v \in \VV} \left\lbrace - \phi^*(-\bm A^* \v) - \psi^*(\v) \right\rbrace, \label{eq:fenchel_dual}
\end{align}
where $\phi^*$ and $\psi^*$ denote the convex conjugates of $\phi$ and $\psi$, respectively.

Let now $\UU = \R^{N_1 \times N_2 \times K}$ and $\VV=\R^{N_1 \times N_2 \times K \times 3}$, such that \eqref{eq:tv_ulog_discrete} is of the form \eqref{eq:fenchel_primal} with
\begin{align}
& \phi(\u) := \chi_{[\Ulow, \Uupp]}(\u),  \qquad \u \in \R^{N_1 \times N_2 \times K},\label{eq:phi} \\
& \psi(\v) := \sum_{i,j,k} \norm{\v_{i,j,k}}, \qquad \v \in \R^{N_1 \times N_2 \times K \times 3}\label{eq:psi}, \\
& \ddiff := \Normnabla_{\x, t} \Normlap: \R^{N_1 \times N_2 \times K} \to \R^{N_1 \times N_2 \times K \times 3}. \label{eq:ddiff}
\end{align}
The next proposition gives an explicit expression for the dual of the TV-ULoG optimization problem \eqref{eq:tv_ulog_discrete}. It is defined with reference to the convex set
\begin{align}
S := \Set{\v \in \R^{N_1 \times N_2 \times K \times 3}}{\norm{\v_{i,j,k}} \leq 1 \text{ for all } i,j,k}. \label{eq:def_s}
\end{align}

\begin{proposition}\label{pro:fenchel_dual}
The Fenchel dual of \eqref{eq:tv_ulog_discrete} is
\begin{equation}
\max_{\v \in S} \min_{\w \in [\Ulow, \Uupp]} \inner{\ddiff^\top \v}{\w}, \label{eq:tv_ulog_dual}
\end{equation}
where $\ddiff$ is as in \eqref{eq:ddiff}.
\end{proposition}
\begin{proof}
Let $\phi$ and $\psi$ be given by \eqref{eq:phi} and \eqref{eq:psi}, respectively. By example 4.2 in \cite{Bec17}, we have
\begin{align}
\phi^*(\bm v) = \max_{\w \in [\Ulow, \Uupp]} \inner{\w} {\v}. \label{eq:conjugate1}
\end{align}
Furthermore, by theorem 4.12 and example 4.4.12 in \cite{Bec17},
\begin{align}
\psi^*(\v) &= \sum_{i,j,k} \chi_{B_1(0)}(\v_{i,j,k}) \notag\\
& = \chi_S(\v), \label{eq:conjugate2}
\end{align}
where the set $S \subset \R^{N_1 \times N_2 \times K \times 3}$ is given by \eqref{eq:def_s}. The proof follows if we plug \eqref{eq:conjugate1} and \eqref{eq:conjugate2} into \eqref{eq:fenchel_dual}.
\end{proof}

The dual problem \eqref{eq:tv_ulog_dual} could be solved with projected subgradient methods \cite{Bec17}. However, a more efficient approach is to consider a smoothed approximation instead, since this allows to apply accelerated methods such as FGP: Given a convex optimization problem of the form
\begin{equation}
\min_{\u \in C_1} ~  \phi(\u), \text{ where } \phi(\u) = \max_{\w \in C_2} \inner{\bm B \u}{\w}, \label{eq:generic_problem}
\end{equation}
where $C_1$ and $C_2$ are bounded closed convex sets, Nesterov \cite{Nes05} proposed to approximate the objective functional $\phi$ by
\begin{align}
\phi_\mu(\bm u) := \max_{\w \in C_2} \set{\inner{\bm B \u}{\w} - \frac{\mu}{2}\norm{\w}^2}. \label{eq:nesterov_smoothing}
\end{align}
The associated optimization problem
\begin{align*}
\min_{\u \in C_1} \phi_\mu(\u)
\end{align*}
is called the Nesterov smoothing of \eqref{eq:generic_problem}. It is a convex problem with smooth objective, and can hence be solved fast using accelerated first-order methods.

The next proposition derives the Nesterov smoothing of the dual problem \eqref{eq:tv_ulog_dual}. The derivation is analogous to the proof of proposition 4.1 in \cite{BecTeb09}. For completeness, we have provided the proof below.

\begin{proposition}\label{pro:smoothed_dual}
The Nesterov smoothing corresponding to \eqref{eq:tv_ulog_dual} is given by
{\scriptsize
\begin{equation}
\min_{\v \in S} \quad \left\{\norm{\frac{1}{\mu} \ddiff^\top \v}^2  -\norm{(I - P_{[\Ulow, \Uupp]})(\frac{1}{\mu} \ddiff^\top \v)}^2 \right\}. \label{eq:smoothed_dual}
\end{equation}
}
\end{proposition}

\begin{proof}
First, \eqref{eq:tv_ulog_dual} is equivalent to
\begin{align*}
\min_{\v \in S} \max_{\w \in [\Ulow, \Uupp]} \inner{-\ddiff^\top \v}{\w}.
\end{align*}
This is of the form \eqref{eq:generic_problem} with $C_1 = S$, $C_2 = [\Ulow, \Uupp]$ and $\bm B = -\ddiff^\top$. Hence, its Nesterov smoothing is
\begin{align*}
& \min_{\v \in K} \quad \phi_\mu(\v), \\
& \phi_\mu(\v) := \max_{\w \in [\Ulow, \Uupp]} \left\lbrace \inner{-\ddiff^\top \v}{\w} - \frac{\mu}{2} \norm{\w}^2\right\rbrace.
\end{align*}
By completing the square, one can show
\begin{equation}
\begin{aligned}
\phi_\mu(\v) = \max_{\w \in [\Ulow, \Uupp]} & \left\{ \frac{\mu}{2} \norm{\frac{1}{\mu} \ddiff^\top \v}^2 \right. \\
& \quad \left. - \frac{\mu}{2} \norm{\w + \frac{1}{\mu} \ddiff^\top \v}^2 \right\}.
\end{aligned}\label{eq:completed_squares}
\end{equation}
Recall the distance minimizing property of the orthogonal projection: For a convex set $C$, we have
\begin{align*}
\min_{\w \in C} \norm{\bm q - \w}^2 = \norm{(I - P_C)(\bm q)}^2. 
\end{align*}
In particular, this means
\begin{equation}
\begin{aligned}
& \max_{\w \in [\Ulow, \Uupp]} \norm{\w + \frac{1}{\mu} \ddiff^\top \v}^2 \notag  \\ & \qquad = \norm{(I - P_{[\Ulow, \Uupp]})(-\frac{1}{\mu}\ddiff^\top \v))}^2. \label{eq:explicit_maximizer}
\end{aligned}
\end{equation}
Inserting \eqref{eq:explicit_maximizer} in \eqref{eq:completed_squares} then yields
\begin{align*}
\phi_\mu(\v) = & \frac{\mu}{2} \left(-\norm{(I - P_{[\Ulow, \Uupp]})(- \frac{1}{\mu}\ddiff^\top \v)}^2 \right. \\
&\quad  \left. + \frac{1}{\mu} \norm{\ddiff^\top \v}^2 \right).
\end{align*}
This finishes the proof.
\end{proof}

The smoothed problem \eqref{eq:smoothed_dual} is now amenable to minimization with the fast gradient projection method given by \autoref{alg:fgp} (cf. \cite{Bec17}).

\begin{algorithm}
\caption{FGP}\label{alg:fgp}
\begin{algorithmic}[1]
\Require{A convex function $\phi: \UU \to \R$ on a Euclidean space $\UU$, a convex set $C \subset \UU$, a steplength $\beta > 0$ and an initial guess $\u_0 \in \UU$.}
\State $\v_0 = \u_0$;
\State $t_0 = 1$;
\For{k = 0,\ldots,K}
\State $\u_{k+1} = P_C(\v_k - \beta\nabla \phi(\v_k))$;
\State $t_{k+1} = \frac{1}{2}(1 + \sqrt{1 + 4 t_k^2})$;
\State $\v_{k+1} = \u_{k+1} + \frac{t_k - 1}{t_k+1} (\u_{k+1} - \u_k)$;
\EndFor
\end{algorithmic}
\end{algorithm}

\begin{remark}\label{rem:regularization_dual_smoothing}
Problem \eqref{eq:smoothed_dual} is identical to the dual problem of CTV-denoising (see proposition 4.1 in \cite{BecTeb09}), except for the definition of the operator $\ddiff$. 
Indeed, problem \eqref{eq:smoothed_dual} is the Fenchel dual of the following regularized version of problem \eqref{eq:tv_ulog_discrete},
\begin{align*}
\min_{\bm u \in \R^{N_1 \times N_2 \times K}} \quad & \Normtv(\Normlap \bm u) + \frac{1}{2 \mu} \norm{\bm u}^2\\
\text{s.t.} \quad &\Ulow \leq \bm u \leq \Uupp.
\end{align*}
\end{remark}

\begin{remark}\label{rem:smoothing_parameter}
Finding a good value for the smoothing parameter $\mu$ in \eqref{eq:nesterov_smoothing} is difficult in practice. In \cite{Nes05}, it is shown that, for all $\u \in \UU$,
\begin{align}
\phi_\mu(\u) \leq \phi(\u) \leq \phi_\mu(\u) + \mu D_2, \label{eq:nesterov_bound}
\end{align}
where $D_2 = \max_{\v \in C_2} \frac{1}{2}\norm{\v}^2$. This means that an accuracy of $\epsilon$ in the objective is ensured if $\mu \leq \epsilon / D_2$. However, the bound \eqref{eq:nesterov_bound} is often too conservative and it can be advantageous to choose $\mu > \epsilon / D_2$, since larger values of $\mu$ allow for larger stepsize and faster convergence.
\end{remark}

\subsubsection{Primal smoothing}\label{sssec:primal_smoothing}

As an alternative to the dual approach described in \autoref{sssec:dual_smoothing}, we can also apply Nesterov-smoothing directly to the primal problem \eqref{eq:tv_ulog_discrete}. 

\begin{proposition}\label{pro:smoothed_primal}
Let $\ddiff$ be given by \eqref{eq:ddiff} and define the Huber loss function \cite{Hub81}
\begin{equation}
 h_\mu: \R^{3} \to \R, \quad h_\mu(\bm v) := \begin{cases} \norm{\bm v} + \frac{\mu}{2},& \text{if } \norm{\bm v} \geq \mu, \\
\frac{\norm{\bm v}^2}{2 \mu},& \text{otherwise}.
\end{cases} \label{eq:huber}
\end{equation}
Then the Nesterov smoothing of the TV-ULoG optimization problem \eqref{eq:tv_ulog_discrete} is
\begin{equation}
\begin{aligned}
\min_{\bm u \in \R^{N_1 \times N_2 \times K}} \quad & \sum_{i,j,k} h_\mu((\ddiff \bm u)_{i,j,k}), \\
\text{s.t.} \quad &\Ulow \leq \bm u \leq \Uupp.
\end{aligned}\label{eq:smoothed_primal}
\end{equation}
\end{proposition}

\begin{proof}
Let
\begin{align}
\phi(\bm u) := \sum_{i,j,k} \norm{(\bm A \bm u)_{i,j,k}}, \label{eq:phi_for_primal}
\end{align}
such that we can write the TV-ULoG optimization problem \eqref{eq:tv_ulog_discrete} as
\begin{align}
\min_{\bm u \in [\Ulow, \Uupp]} \phi(\bm u). \label{eq:tv_ulog_rewrite}
\end{align}
Using the identity
\begin{align*}
\norm{\v} = \max_{\w \in B_1(0)} \inner{\w}{\v},
\end{align*}
in \eqref{eq:phi_for_primal} yields
\begin{align*}
\phi(\u) &= \sum_{i,j,k} \max_{\w_{i,j,k} \in B_1(0)} \inner{\w_{i,j,k}}{(\ddiff \u)_{i,j,k}} \\
& = \max_{\w \in S}\inner{\w}{\bm A \bm u},
\end{align*}
where $S$ is again given by \eqref{eq:def_s}. Hence, \eqref{eq:tv_ulog_rewrite} is of the form \eqref{eq:generic_problem} with $C_1 = [\Ulow, \Uupp]$ and $C_2 = S$, which means that the Nesterov-smoothed objective is given by (cf. \eqref{eq:nesterov_smoothing})
\begin{align}
\phi_\mu(\bm u) & = \max_{\w \in S} \left(\inner{\w}{\u} - \frac{\mu}{2} \norm{\w}^2 \right) \notag \\
&= \sum_{i,j,k} \max_{\w_{i,j,k} \in B_1(0)} \left(\inner{\w_{i,j,k}}{(\bm A \u)_{i,j,k}} - \frac{\mu}{2} \norm{\w_{i,j,k}}^2 \right). \label{eq:nesterov_smoothed_primal}
\end{align}
It is straightforward to show (e.g. using Lagrange multipliers) that, for any vector $\v$,
\begin{align}
\max_{\w \in B_1(0)} \left( \inner{\v}{\w} - \frac{\mu}{2} \norm{\w}^2 \right) = h_\mu(\v), \label{eq:max_is_huber}
\end{align}
where $h_\mu$ is defined in \eqref{eq:huber}. Using \eqref{eq:max_is_huber} in \eqref{eq:nesterov_smoothed_primal} proves the proposition.
\end{proof}
Problem \eqref{eq:smoothed_primal} is a smooth optimization problem with bound constraints for which there exists a huge number of optimization methods, such as the previously introduced FGP method or the popular L-BFGS-B method \cite{ByrLuNocZhu95,ZhuByrLuNoc97,MorNoc11}.

\subsubsection{Interior-point method}\label{sssec:socp}

Many problems involving TV-regularization, in particular $L^1$-TV denoising and CTV-denoising, can be formulated equivalently as second-order cone programs (SOCP) and then be solved with interior-point methods \cite{YinGolOsh07}. This strategy has the advantage that it does not require additional smoothing. Interior-point methods are also known to be much more robust against ill-conditioning of the KKT system. They require the solution of a large linear system in every step, which is why they sometimes do not scale well to larger problems. However, if the special structure of the linear system can be exploited, interior-point methods can yield state-of-the-art performance \cite{GolYin05}.

We call an optimization problem a SOCP if it can be written in the following form \cite{BoyVan04}:
\begin{equation}
\begin{aligned}
\min_{\v \in \R^n} \quad & \inner{\bm \xi}{\v} \\
\text{s. t. } \quad & \norm{\bm B_i \v + \bm c_i} \leq \inner{\bm d_i}{\v} + \eta_i, \quad i \in [m],\\
& \bm H \v = \bm h,
\end{aligned} \label{eq:socp_standard_form}
\end{equation}
where $\bm \xi, \bm d_i \in \R^n$, $\bm B_i \in \R^{n_i \times n}$,  $\bm c_i \in \R^{n_i}$, $\eta_i \in \R$, $\bm H \in \R^{p \times n}$, and $\bm h \in \R^p$. 

The next proposition states how the TV-ULoG optimization problem \eqref{eq:tv_ulog_discrete} can be brought into the form \eqref{eq:socp_standard_form}. To prepare the proof, we bring the primal problem in flattened form so that we work with vectors instead of three-dimensional arrays. To this end, let $N := N_1 \cdot N_2 \cdot K$ and define the flattening operator
\begin{align*}
& \flatten: \R^{N_1 \times N_2 \times K} \to \R^{N}, \\
& \flatten(\bm u)_{\sigma(i,j,k)} = u_{i,j,k},
\end{align*}
where
\begin{align*}
& \sigma: [N_1] \times [N_2] \times [K] \to [N], \\
& \sigma(i, j, k) = i N_2 K + j K + k
\end{align*}
is an enumeration of $[N_1] \times [N_2] \times [K]$. Let
\begin{align}
\zlow := \flatten(\Ulow), \quad \zupp := \flatten(\Uupp). \label{eq:flat_bounds}
\end{align}
Then, one can find matrices $\ddiffflat_1,\ldots, \ddiffflat_N \in \R^{3 \times N}$ such that
\begin{align}
(\ddiffflat_{\sigma(i,j,k)} \flatten(\u))_r = (\ddiff \u)_{i,j,k,r} \label{eq:flat_operator}
\end{align}
for all $\u \in \R^{N_1 \times N_2 \times K}$, $r \in [3]$ and $(i,j,k) \in [N_1] \times [N_2] \times [K]$. Under these definitions, it is easy to see that the TV-ULoG optimization problem \eqref{eq:tv_ulog_discrete} is equivalent to
\begin{equation}
\begin{aligned}
\min_{\bm z \in \R^N} \quad & \sum_{\ell=1}^N \norm{\ddiffflat_\ell \bm z} \\
\text{s. t.} \quad & \zlow \leq \bm z \leq \zupp,
\end{aligned}\label{eq:vectorized_tv_ulog}
\end{equation}
where $\bm z = \flatten(\u)$.

\begin{proposition}\label{pro:socp}
The optimization problem \eqref{eq:vectorized_tv_ulog} is equivalent to the SOCP \eqref{eq:socp_standard_form}, under the identifications
{\scriptsize%
\begin{equation}
\begin{aligned}
& n = 2 N, \quad m = 3N, \quad \bm \xi = \begin{bmatrix}
\bm{0}_N \\ \bm{1}_N
\end{bmatrix}, \\
& \bm B_\ell = \bm{0}_{1\times 2N}, \quad \bm B_{N + \ell} = \bm{0}_{1 \times n}, \quad \bm B_{2N + \ell} = \begin{bmatrix} \ddiffflat_\ell & \bm{0}_{3 \times n} \end{bmatrix}, \\
& \bm c_\ell = 0, \quad \bm c_{N + \ell} = 0, \quad \bm c_{2N + \ell} = \bm{0}_3, \\
& \bm d_\ell = \bm e_\ell^{n}, \quad \bm d_{N + \ell} = - \bm e_\ell^{n}, \quad \bm d_{2N + i} = \bm e_{N + \ell}, \\
& \eta_\ell = -z^\mathrm{low}_\ell, \quad \eta_{N + \ell} = z^\mathrm{upp}_\ell, \quad \eta_{2N + \ell} = 0, \quad \ell \in [N], \\
& \bm H = \bm{0}_{1 \times n} , \quad \bm h = 0.
\end{aligned}\label{eq:transform_to_socp}
\end{equation}}%
In particular, if $\bar{\v} \in \R^{2N}$ is a minimizer of \eqref{eq:socp_standard_form}, then $\bar{\u}$ given by
\begin{align*}
\bar u_{i,j,k} = \bar v_{\sigma(i,j,k)}
\end{align*}
is a minimizer of the TV-ULoG optimization problem \eqref{eq:tv_ulog_discrete}.
\end{proposition}

\begin{proof}
Note that \eqref{eq:vectorized_tv_ulog} is equivalent to
\begin{equation}
\begin{aligned}
\min_{\bm z \in \R^N, \bm q \in \R^N} \quad & \sum_{\ell=1}^N q_i \\
\text{s. t.} \quad & \zlow \leq \bm z \leq \zupp, \\
& \norm{\ddiffflat_\ell \bm z} \leq q_\ell, \quad \ell \in [N].
\end{aligned}\label{eq:bound_trick}
\end{equation}
Let us combine the optimization variables $\bm z$ and $\bm q$ in a single vector $\v = [\bm z, \bm q] \in \R^{2 N}$, such that \eqref{eq:bound_trick} becomes
\begin{equation}
\begin{aligned}
\min_{\v \in \R^{2N}} \quad & \sum_{\ell=1}^N v_{N + \ell}, \\
\text{s. t.} \quad & z^\mathrm{low}_\ell \leq v_\ell \leq z^\mathrm{upp}_\ell, \\
& \norm{\begin{bmatrix} \ddiffflat_\ell & \bm{0}_{3\times N} \end{bmatrix} \v} \leq v_{N + \ell}, \quad \ell \in [N].
\end{aligned}\label{eq:bound_trick2}
\end{equation}
It is now straightforward to check that if we plug the identifications \eqref{eq:transform_to_socp} into the SOCP standard form \eqref{eq:socp_standard_form}, we obtain precisely \eqref{eq:bound_trick2}.
\end{proof}

The resulting SOCP can then be solved efficiently with existing interior-point solvers (see also \autoref{subsubsec:comparison_optimization_methods}), exploiting the sparse structure of the discretized differential operator $\ddiff$.

\subsection{Extraction of blob regions}\label{subsec:extraction}

Let $\Blanket$ be a numerical solution of the TV-ULoG optimization problem \eqref{eq:tv_ulog_discrete}. Ideally, $\Normlap \Blanket$ is piecewise constant such that it attains its local minima on index sets $\mathcal M_1, \ldots, \mathcal M_S \subset [N_1] \times [N_2] \times [K]$ which quantify our uncertainty with respect to the blobs of the uncertain image (see \autoref{sec:tube_methods}). 

However, due to numerical errors, $\Normlap \Blanket$ will typically only be approximately piecewise constant or exhibit artifacts such as staircasing. For this reason, we use the following thresholding procedure to extract the desired regions.
\begin{enumerate}
\item Let $\bm a := \Normlap \Blanket \in \R^{N_1 \times N_2 \times K}$.
\item Detect the local minimizers $\bm m_s, \ldots, \bm m_S \in [N_1] \times [N_2] \times [K]$ of $\bm a$.
\item For each local minimizer $m_s = (i_s, j_s, k_s)$, detect the corresponding plateau, i.e. the largest connected component $\mathcal M_s \in [N_1] \times [N_2] \times [K]$ with $m_s \in \mathcal M_s$ such that
\begin{align*}
a_{i,j,k} \leq r a_{i_s,j_s,k_s} \quad \text{for all } (i,j,k) \in \mathcal M_s.
\end{align*}
\end{enumerate}
Here, $r \in (0, 1)$ is a relative threshold that determines the size of the resulting regions. If we choose $r$ closer to 1, the resulting regions will be tighter but the results of the method will be less robust against numerical errors.

\subsection{Visualization}\label{subsec:visualization}

The result of the extraction step are regions $\mathcal M_1, \ldots, \mathcal M_S \subset [N_1] \times [N_2] \times [K]$. Since these are sets in a discrete three-dimensional space, they are difficult to visualize directly. The extent of the regions along the first two axes (the spatial domain) expresses uncertainty with respect to position of the corresponding blob,  while the extent along the third axis (the scale)  corresponds to uncertainty in the size of the corresponding blob (see the middle panel in \autoref{fig:stellar_recovery} below).

We suggest to visualize the uncertainty in scale and position by introducing two projections on the pixel grid $[N_1] \times [N_2]$ which are easy to visualize as images.

The first projection is the direct projection on the image domain,
\begin{equation}
\scriptscriptstyle
\begin{aligned}
& \Pi_1: 2^{[N_1] \times [N_2] \times [K]} \to 2^{[N_1] \times [N_2]}, \\
& \Pi_1 (\mathcal M) :=  \Set{(i,j)}{\exists k \in [K]: (i,j,k) \in \mathcal M}.
\end{aligned}\label{eq:projection1}
\end{equation}
The set $\Pi_1(\mathcal M)$ is a two-dimensional region that contains the centers of all blobs in $\mathcal M$, and can thus be used to visualize the uncertainty in the center position of the uncertain blob.

The second projection is motivated by the visualization for the discrete Laplacians-of-Gaussians method, where a point $(i, j, k) \in [N_1] \times [N_2] \times [K]$ is visualized by a (discretized) circle $B_{\sqrt{2 t_k}}(i,j)$.  Taking the possibly different grid sizes $h_1$ and $h_2$ into account, the set $B_r(i,j) \subset [N_1] \times [N_2] \times [K]$ is defined as
{\scriptsize
\begin{align*}
B_r(i,j) = \Set{(i', j')}{\sqrt{h_1^2(i' - i)^2 + h_2^2 (j' - j)^2} \leq r}.
\end{align*}}
We then define the projection of a set $\mathcal M \subset [N_1] \times [N_2] \times [K]$ onto the union over all circles corresponding to points in $\mathcal M$ by
\begin{equation}
\scriptscriptstyle
\begin{aligned}
& \Pi_2: 2^{[N_1] \times [N_2] \times [K]} \to 2^{[N_1] \times [N_2]}, \\
& \Pi_2 (\mathcal M) :=  \bigcup_{(i, j, k) \in \mathcal M} B_{\sqrt{2 t_k}}(i, j).
\end{aligned}\label{eq:projection2}
\end{equation}
Together, $\Pi_1(\mathcal M)$ and $\Pi_2(\mathcal M)$ allow to visualize the uncertainty in the blob center and the blob scale within a set $\mathcal M$. An example of this visualization is provided in \autoref{fig:stellar_recovery} (middle panel). \autoref{fig:deconvolution_blobs} gives an example how this visualization can be used on one-dimensional signals.

\subsection{Further remarks}

As mentioned in the beginning of \autoref{subsec:implementation_tv}, the TV-ULoG optimization problem \eqref{eq:tv_ulog_discrete} has many parallels to CTV-denoising. For example, \cite{WeiBlaAub09} considers constrained-TV regularization and proposes a Nesterov smoothing approach similar to the one we describe in \autoref{sssec:primal_smoothing}. The TV-ULoG optimization problem could also be attacked with ADMM, similar to the method presented in \cite{ChaTaoYua13}.

However, there are two significant differences between the TV-ULoG optimization problem and  the CTV-denoising problem \eqref{eq:ctv_denoising}, apart from the higher dimensionality:
First, the CTV-denoising problem depends through the TV-term on the first derivative of the image, while in the TV-ULoG optimization problem we take the total variation of the normalized Laplacian, which means that we depend on the third-order derivatives. The resulting discrete differential operator $\ddiff$ has a high condition number, which makes the TV-ULoG optimization problem more difficult to solve numerically. This problem is amplified by the fact that there is no denoising term in the TV-ULoG optimization problem \eqref{eq:tv_ulog_discrete}, since this term otherwise has a regularizing effect. These differences could explain the problems of the proposed first-order methods we observed in our numerical experiments (see \autoref{subsubsec:comparison_optimization_methods}).

Alternatively, problem \eqref{eq:tv_ulog_discrete} could be solved with a semi-smooth Newton method \cite{HinItoKun02}, which is likely more accurate than the first-order methods. Such an approach would scale similar as the interior-point strategy.

A final option that we did not investigate further are graph-cut methods \cite{Cha05, DarSig06}, which could potentially be very efficient for our approach, since the result of the TV-ULoG method should be robust against the quantization error. However, such a method would be more difficult to implement.

\section{Numerical experiments}\label{sec:numerics}

In this section, we illustrate the tube-based blob detection methods introduced in this paper on two Bayesian inverse problems. 

We start with a one-dimensional problem which mostly serves didactical purposes, since the scale-space representation of a one-dimensional signal is two-dimensional and can therefore be visualized as an image. This allows us to further illustrate and discuss the ideas underlying our approach. 

As second example we consider a more challenging two-dimensional imaging problem from stellar dynamics since this application was a main motivation for the present work.

The numerical experiments were implemented in Python 3.10. The source code is available from the GitHub repository \href{https://github.com/FabianKP/tvulog}{github.com/FabianKP/tvulog}, which also provides an exact list of the used packages. The reported computation times were measured on a PC with 64 GB RAM and 8 3.6-GHz Intel i7-7700 CPUs.

\subsection{One-dimensional deconvolution}\label{subsec:deconvolution}

\subsubsection{Problem setup}\label{sssec:deconvolution_setup}

\begin{figure}[h]
\centering
\includegraphics[width=1\columnwidth]{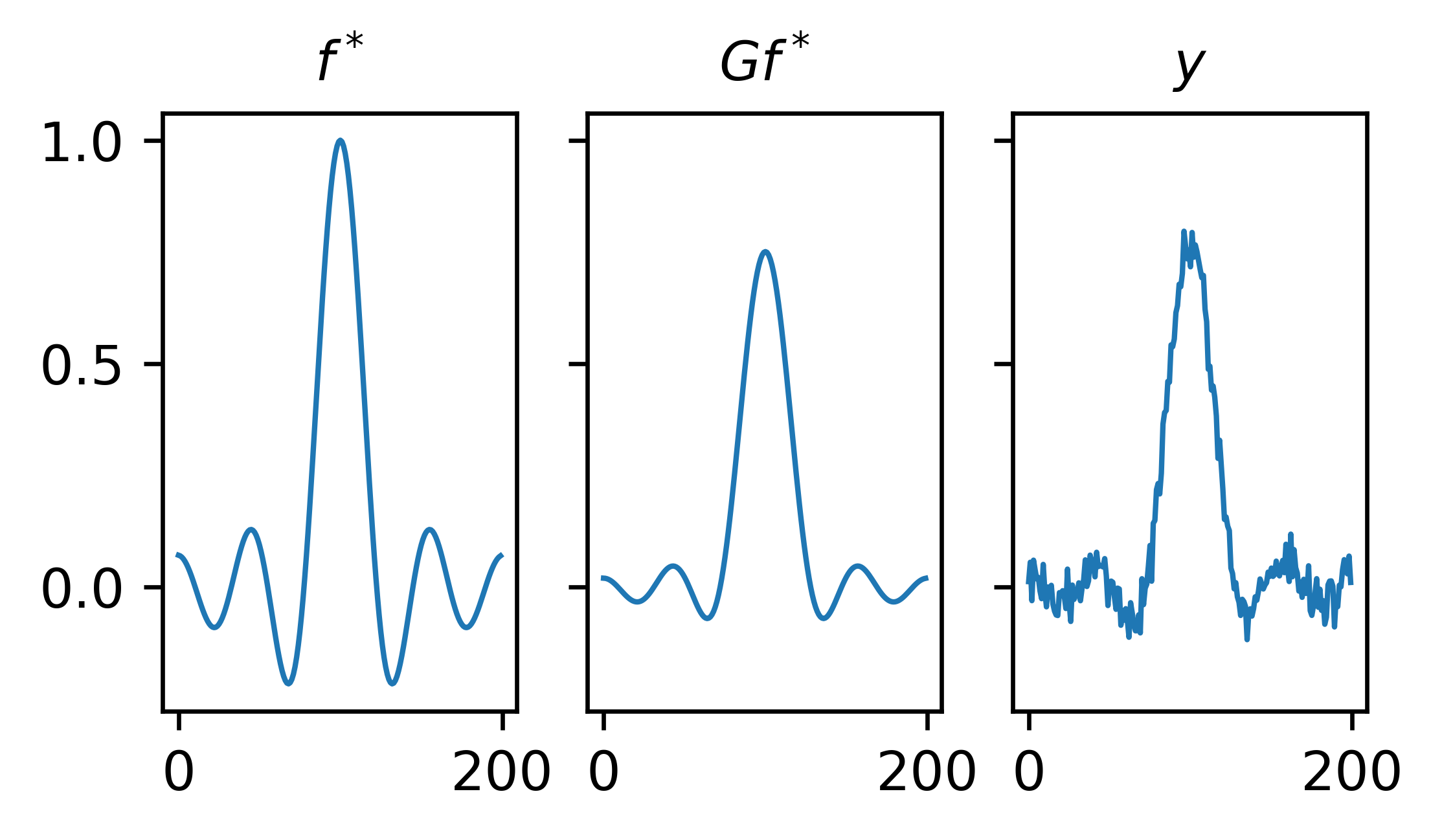}
\caption{Setup of the deconvolution problem. Left panel: The ground truth $\f^*$. Middle panel: The convolved ground truth $\Fwd \f^*$. Right panel: The noisy data $\bm y = \Fwd \f^* + \bm w^*$.}\label{fig:deconvolution_setup}
\end{figure}

We tested the proposed method on the problem of blob detection in one-dimensional Bayesian deconvolution. That is, we considered the task of identifying blobs in a one-dimensional discrete signal, modelled as realization of a random vector $\F$, using data $\Y$ which is given by the noisy convolution
\begin{align*}
\Y = \Fwd \F + \bm W,
\end{align*}
where $\Fwd: \R^N \to \R^N$ is a convolution operator, and $\bm W$ is zero-mean Gaussian noise. While we described the proposed methodology only for the case of two-dimensional signals (images), the adaptation to one-dimensional signals is completely straightforward, and we skip it to avoid repetition.

To define a full Bayesian model for the deconvolution problem, we assign a zero-mean Gaussian prior on $\F$,
\begin{align}
p_\F = \normal(\bm{0}, \bm \Sigma), \label{eq:deconvolution_prior}
\end{align}
where $\bm \Sigma \in \R^{N \times N}$ is the prior covariance matrix. In our experiments, we used a prior covariance $\bm \Sigma$ corresponding to a Gaussian random Markov field prior. 

Assuming that the noise $\bm W$ is uncorrelated with constant standard deviation $\gamma$, we arrive at the likelihood
\begin{align}
p_{\Y | \F}(\cdot \given \f) = \normal(\Fwd \f, \gamma^2 \Idmat), \quad \f \in \R^N. \label{eq:deconvolution_likelihood}
\end{align}
Combining \eqref{eq:deconvolution_prior} and \eqref{eq:deconvolution_likelihood} via Bayes theorem yields the posterior density
{\scriptsize%
\begin{align*}
p_{\F | \Y}(\f \given \y) \propto \exp\left(-\frac{1}{2 \gamma^2} \norm{\Fwd \f - \y}^2 -\frac{1}{2} \norm{\bm \Sigma^{-1/2} \f}^2 \right).
\end{align*}}

\subsubsection{Simulation}

\begin{figure}[h]
\centering
\includegraphics[width=0.9\columnwidth]{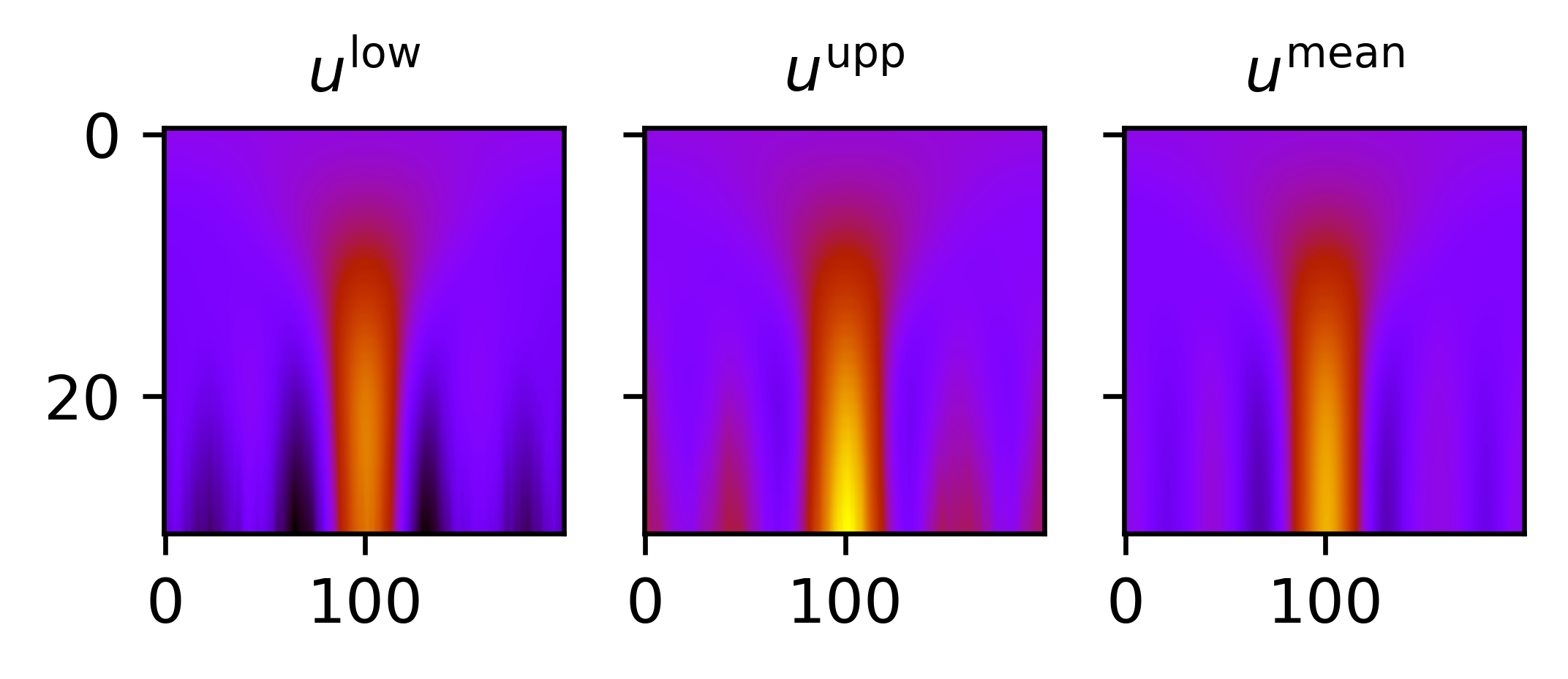}
\caption{The lower and upper bound of the scale-space tube $[\Ulow, \Uupp]$ for the one-dimensional deconvolution problem. The scale-space representation of the posterior mean is plotted in the right panel for comparison.}\label{fig:deconvolution_tube}
\end{figure}

\begin{figure}[h]
\centering
\includegraphics[width=0.8\columnwidth]{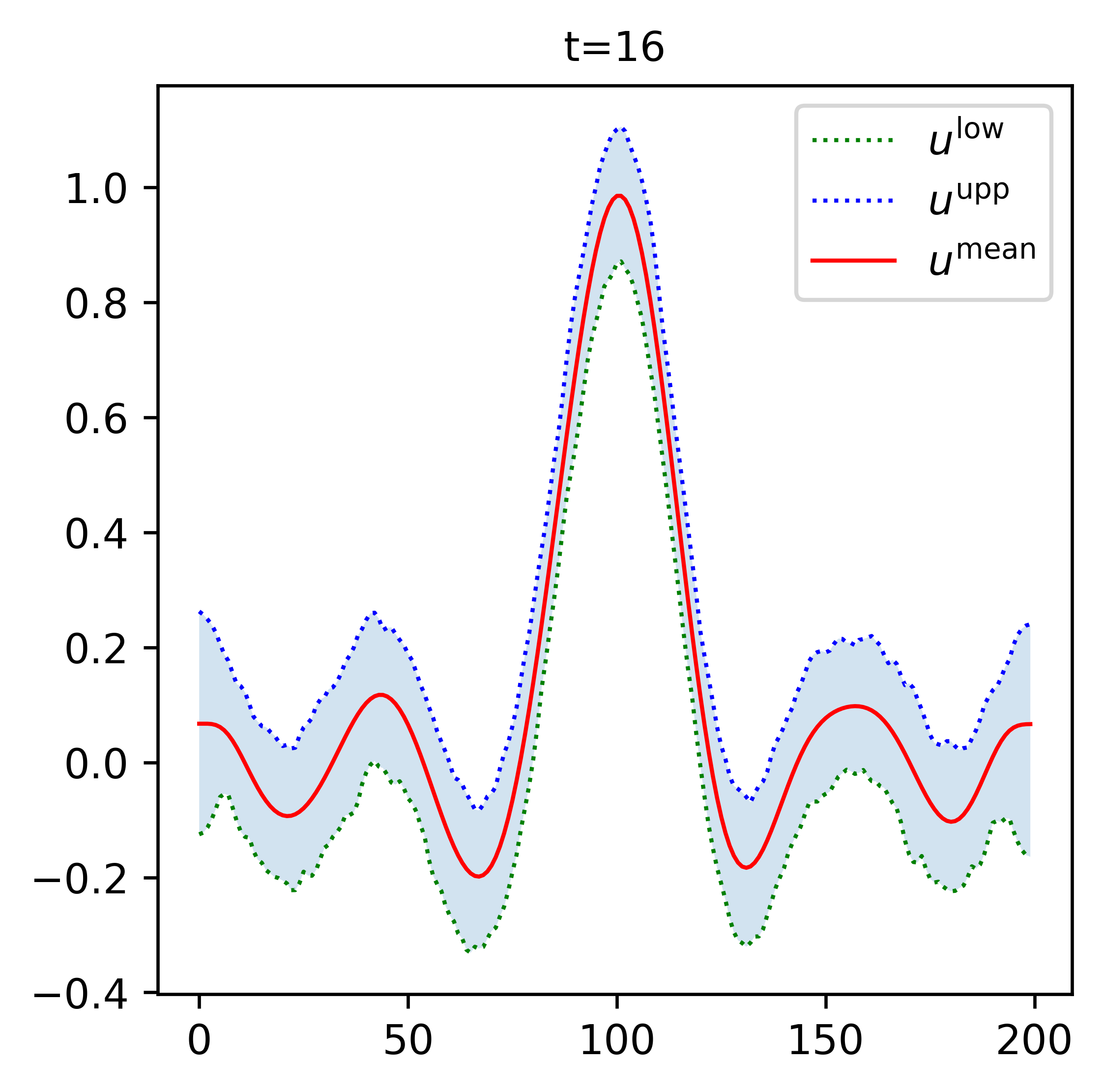}
\caption{Slice through the scale-space tube $[\Ulow, \Uupp]$ for a fixed scale.}\label{fig:deconvolution_tube_slice}
\end{figure}

For our numerical experiment, we used a sinusoidal ground truth $\f^* \in \R^{N \times N}$ (see \autoref{fig:deconvolution_setup}) with $N=200$, from which we simulated noisy data $\bm y = \Fwd \f^* + \bm w$, where $\bm w$ was generated from $\normal(\bm{0}, \gamma^2 \bm I)$ with $\gamma=0.03$. 

We then generated 10 000 MCMC samples using the Linear-RTO sampler provided by the CUQIpy Python package \cite{RiiAlgUriChrAfk23}. Using these MCMC samples, we computed a scale space tube $[\Ulow, \Uupp]$ for the uncertain scale-space representation $\U = \bm \Phi \F$ using the heuristic method described in \autoref{app:tubes_from_mcmc} for the crediblity parameter $\alpha = 0.05$ (corresponding to 95\% credibility). We used exponentially increasing scales (cf. \eqref{eq:exponential_scales})
\begin{equation}
\begin{aligned}
t_k & = b^{k-1} t_\text{min}, \qquad k \in [K],\\
b & = \left(\frac{t_\text{max}}{t_\text{min}}\right)^\frac{1}{K-1},
\end{aligned}\label{eq:discrete_scales}
\end{equation}
with $K=30$, $t_\text{min} = 1$ and $t_\text{max} = 70^2$. The two-dimensional objects $\Ulow$ and $\Uupp$ are visualized in \autoref{fig:deconvolution_tube}. Since it is hard to see the difference between the lower and uppper bound with the naked eye, we have plotted a horizontal slice (that is, a slice for a fixed scale) in \autoref{fig:deconvolution_tube_slice}. For reference, we also computed a point estimate for the signal of interest $\f$ in form of the posterior mean, given by
\begin{align*}
\f^\text{mean} = \frac{1}{S} \sum_{s=1}^S \f^{(s)},
\end{align*}
where $(\f^{(s)})_{s=1}^S$ are the computed MCMC samples. We denote the scale-space representation of $\f^\text{mean}$ by $\bm u^\text{mean}$.

\subsubsection{Results}

\begin{figure*}[h]
\includegraphics[width=2\columnwidth]{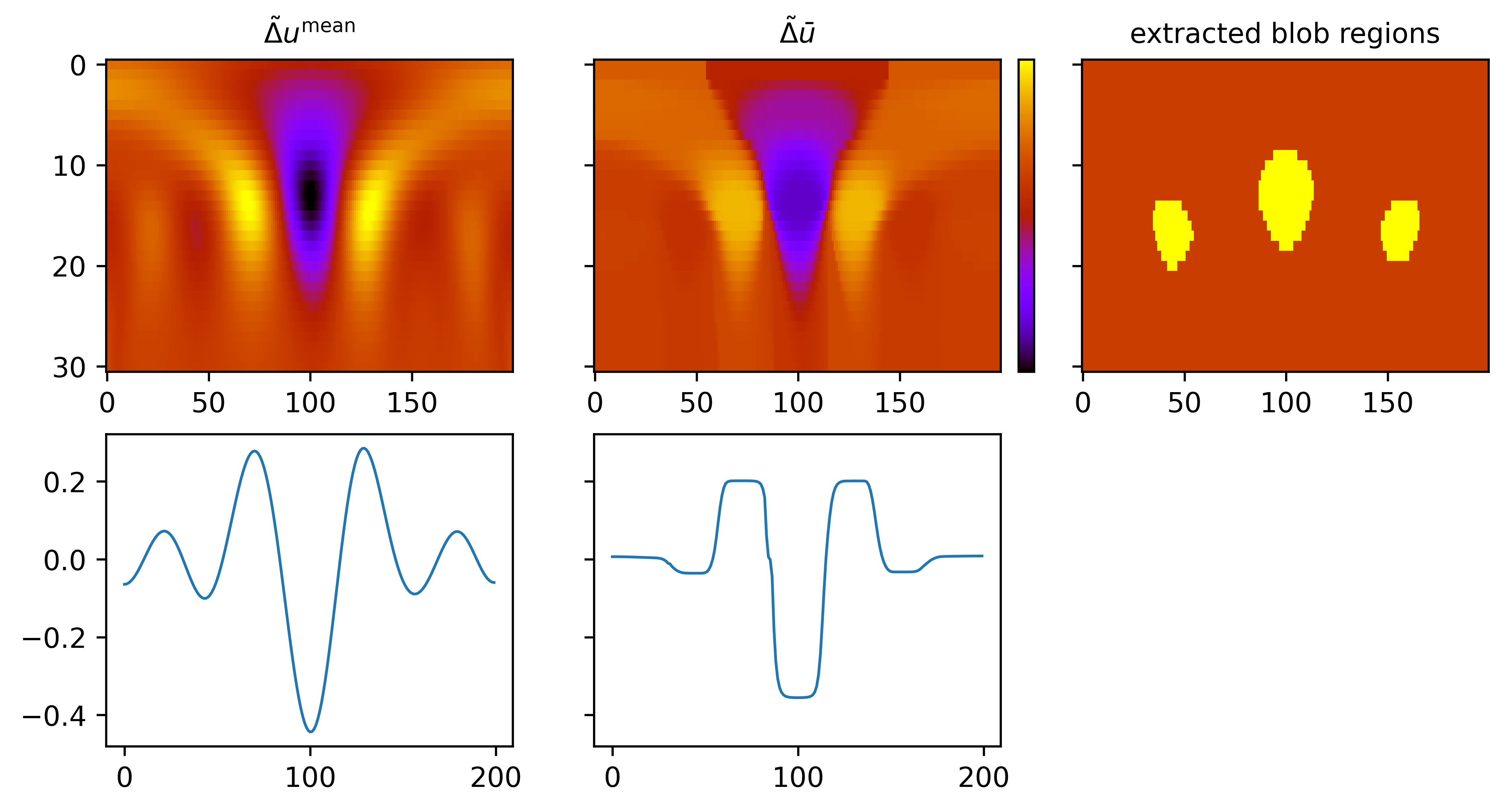}
\caption{The result of the TV-approach for the one-dimensional deconvolution problem. First row, left panel: the scale-normalized Laplacian of the posterior mean; middle panel: the scale-normalized Laplacian of the minimizer of \eqref{eq:tv_ulog_discrete}; right panel: the blob regions in scale space extracted from $\normlap \bar{u}$. Second row: Corresponding horizontal slice for fixed scale $t_k$, $k=15$.}\label{fig:tv_minimizer}
\end{figure*}

\begin{figure}[h]
\centering
\includegraphics[width=0.8\columnwidth]{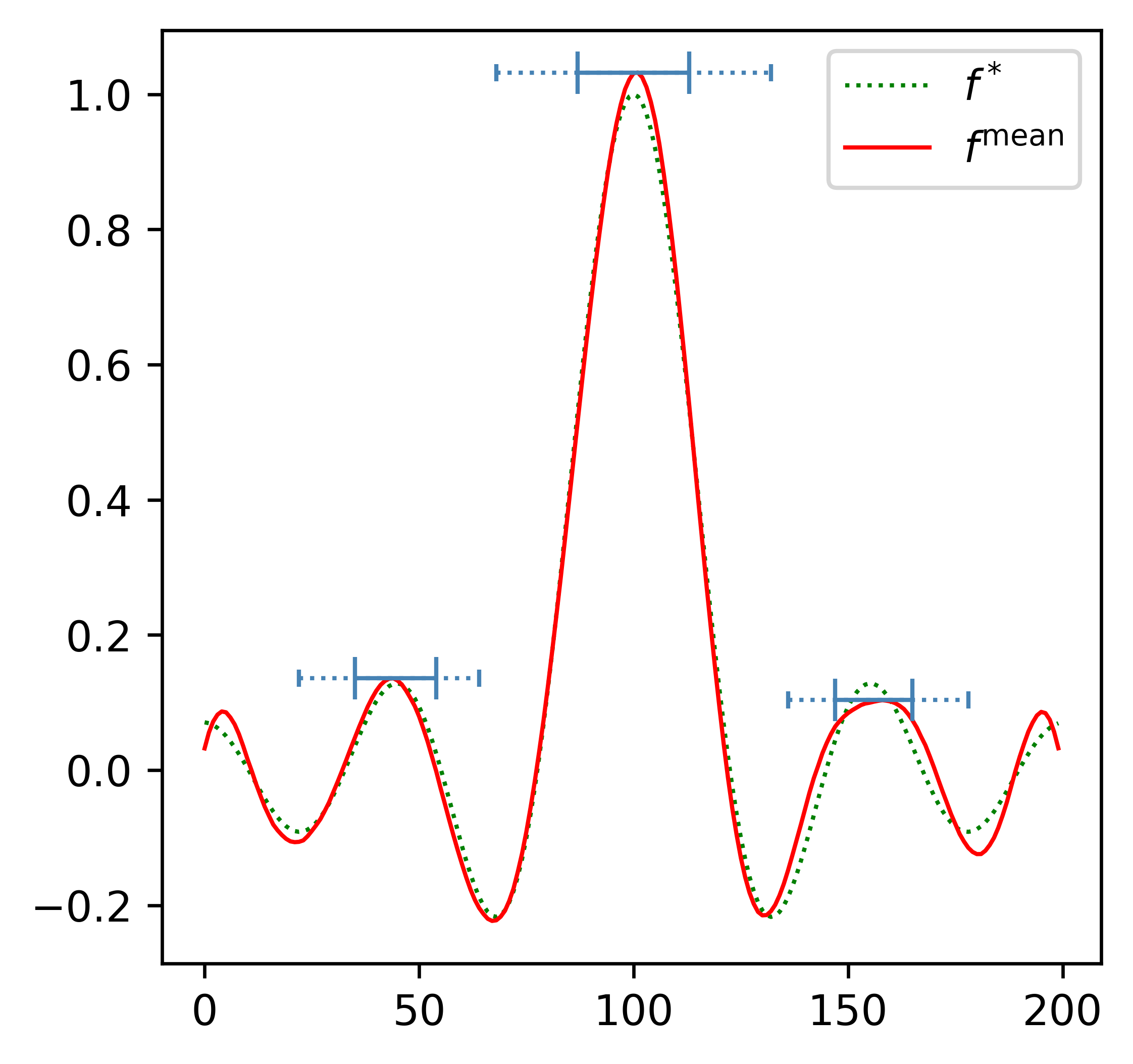}
\caption{Plot of the posterior mean $\f^\text{mean}$ for the one-dimensional deconvolution problem together with the uncertain blobs. The projected blob sets are visualized by blue horizontal bars, where the solid bar indicates the center projection and the dotted bar indicates the scale projection. The ground truth $\f^*$ (dotted line) is also plotted for reference.}\label{fig:deconvolution_blobs}
\end{figure}

To solve the optimization problem \eqref{eq:tv_ulog_discrete}, we used the interior-point approach since it was by far the most efficient (see \autoref{subsubsec:comparison_optimization_methods} below). The normalized Laplacian of the computed minimizer is plotted in \autoref{fig:tv_minimizer} (middle panel). Compared to the normalized Laplacian of the posterior mean (left panel), the scale-normalized Laplacian is approximately piecewise constant and attains local minima on four clearly separated regions, which were extracted using the thresholding procedure described in \autoref{subsec:extraction} (right panel). Since these regions are difficult to make out with the bare eye, a horizontal slice (that is, for fixed scale) is plotted in the second row of \autoref{fig:tv_minimizer}.

In \autoref{fig:deconvolution_blobs}, we plot the extracted blob regions using the procedure described in \autoref{subsec:visualization}, together with the posterior mean and the ground truth for comparison. The horizontal solid bars are obtained from the projection $\Pi_1$ \eqref{eq:projection1}. That is, they indicate the intervals in which we expect the blob \emph{centers} to lie with 95\%-confidence. Similarly, the dotted bars are obtained from the projection $\Pi_2$ \eqref{eq:projection2} and indicate the maximal extent of the uncertain blob.

\subsection{Integrated-light stellar population recovery}\label{subsec:stellar_recovery}

\subsubsection{Problem setup}

Next, we revisit the problem of integrated-light stellar population recovery. This is a constrained linear imaging inverse problem of the form
\begin{align*}
\Y = \Fwd \F + \bm W, \qquad \bm F \geq 0,
\end{align*}
where $\F$ is a two-dimensional non-negative density function (modelled as random image), $\Y$ is a measured light spectrum and $\bm W$ is zero-mean uncorrelated Gaussian noise. The observation operator $\Fwd: \R^{N_1 \times N_2} \to \R^M$ is the discretization of an integral operator that models how the density $\bm F$ influences the spectrum. We do not discuss the details of this problem and the Bayesian modelling and instead refer the reader to the previous work \cite{ParJetBoeAlfSch23a}.

\subsubsection{Simulation}

For our numerical experiment, we simulated a realization $\y$ of the noisy spectrum $\Y$ from a ground truth $\F = \f^*$ (\autoref{fig:stellar_recovery}, top panel) and generated 10 000 posterior samples (after 5000 burn-in iterations) using the SVD-MCMC method described in \cite[section 4]{ParJetBoeAlfSch23a}. As in the one-dimensional deconvolution example, we computed a scale-space tube $[\Ulow, \Uupp]$ using the method described in \autoref{app:tubes_from_mcmc} for the credibility parameter $\alpha = 0.05$ and discrete scales given by \eqref{eq:discrete_scales} with $K=16$, $t_\text{min} = 1$ and $t_\text{max} = 30^2$.

For reference, we also computed a point estimate for the signal of interest $\F$, this time in form of the maximum-a-posteriori estimate $\f^\text{MAP}$.

\subsubsection{Results}

\begin{figure}[!h]
\includegraphics[width=\columnwidth]{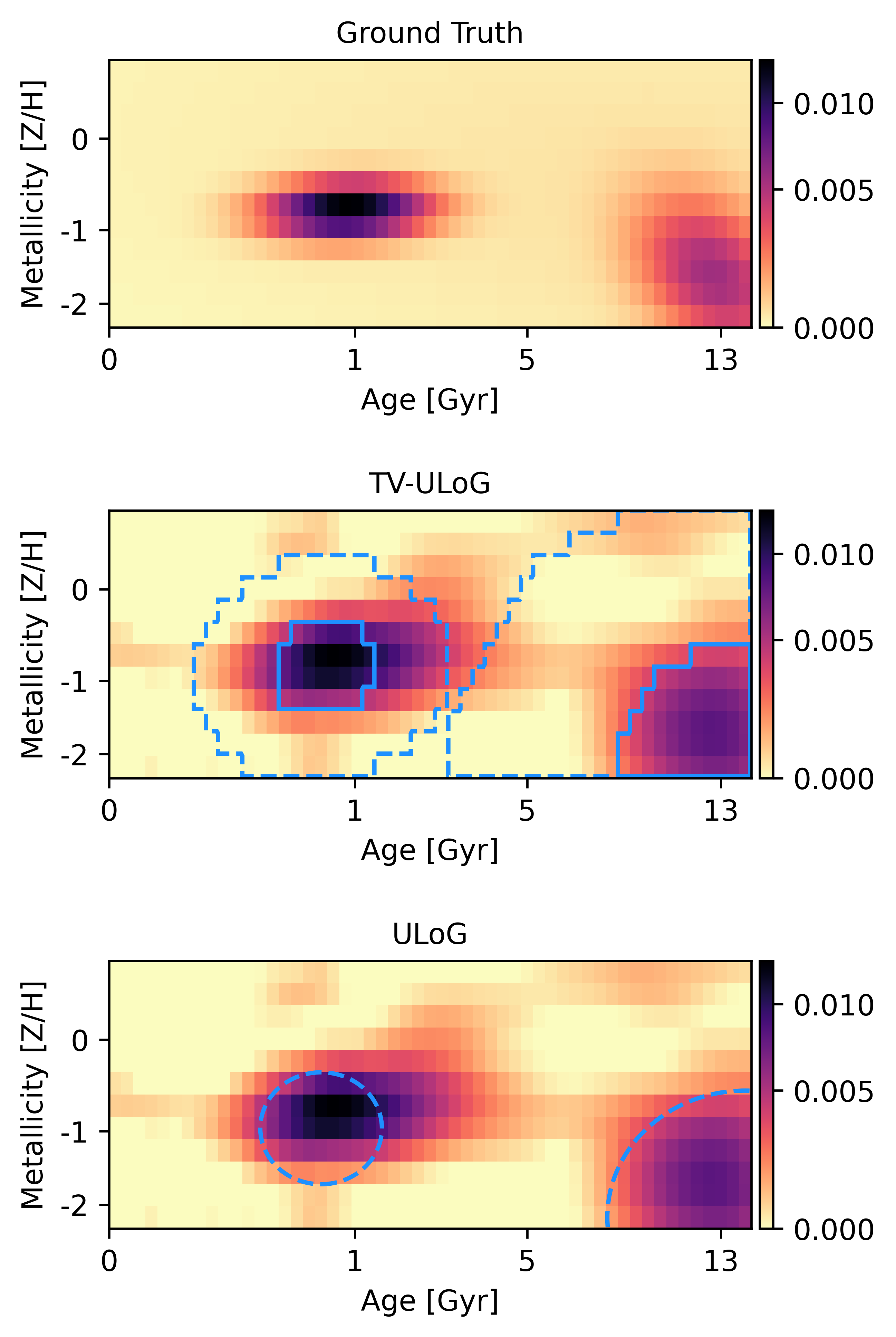}
\caption{Results of ULoG and TV-ULoG for integrated-light stellar population recovery. Top panel: Ground truth $\f^*$ from which the mock data was generated. Middle panel: MAP estimate $\f^\text{MAP}$ together with projection of blob sets. The two projections $\Pi_1$ and $\Pi_2$ are indicated by solid and dashed blue lines, respectively. Bottom panel: Results of ULoG superimposed on the MAP estimate.}\label{fig:stellar_recovery}
\end{figure}

We computed a minimizer of the associated optimization problem \eqref{eq:tv_ulog_discrete} using the interior-point method, since it was again the most efficient (see also \autoref{subsubsec:comparison_optimization_methods}). Two blob regions were extracted and visualized using the procedures described in \autoref{subsec:extraction} and \autoref{subsec:visualization}. In \autoref{fig:stellar_recovery} (middle panel), the projected blob regions are plotted together with the MAP estimate $\f^\text{MAP}$. 

\subsubsection{Comparison with ULoG}

The bottom panel of \autoref{fig:stellar_recovery} shows the result of the ULoG method (see \autoref{subsec:ulog}). Recall that in the ULoG method, we compute a minimizer $\Blanket$ of
\begin{equation*}
\begin{aligned}
\min_{\bm u \in \R^{N_1 \times N_2 \times K}} \quad & \norm{\Normlap \bm u}^2 \\
\text{s.t.} \quad &\Ulow \leq \bm u \leq \Uupp,
\end{aligned}
\end{equation*}
and then determine the local minimum points of $\normlap \Blanket$. A local minimum point $(i,j,k)$ of $\Blanket$ is then visualized by a dashed blue circle with center $(i,j)$ and radius $\sqrt{2 t_k}$ (see also \autoref{ex:prototypical_blob}). In contrast, the proposed TV-ULoG method yields a representative $\Blanket$ such that $\normlap \Blanket$ attains its local minima on connected regions (see also \autoref{subsec:tv_approach}), which are visualized in the middle panel of \autoref{fig:stellar_recovery} using the method described in \autoref{subsec:visualization}.

We see that both methods detect two blobs with confidence. However, the projected blob regions obtained from TV-ULoG allow for a better localization of the blob center. More importantly, the projected regions have a clear interpretation: The inner regions contain the center of the uncertain blob with 95\%-probability. The outer regions contain the corresponding blob circles. In contrast, the dashed circles provided by ULoG have a less clear interpretation since they mix scale and uncertainty information (see the discussion after \eqref{eq:old_cost_function}).

\subsection{Comparison of optimization methods}\label{subsubsec:comparison_optimization_methods}

Both for the one-dimensional deconvolution example and the stellar recovery example, we compared the performance of the optimization strategies proposed in \autoref{subsec:implementation_tv} for the solution of the optimization problem \eqref{eq:tv_ulog_discrete}, namely the dual-smoothing approach (\autoref{sssec:dual_smoothing}), the primal-smoothing approach (\autoref{sssec:primal_smoothing}) and the interior-point method (\autoref{sssec:socp}).

Both for the dual- and primal-smoothing approach we used the FGP method (see \autoref{alg:fgp}) to solve the smoothed optimization problem, while we used ECOS \cite{DomChuBoy13} for the interior-point approach.

The results of our comparison are plotted in \autoref{fig:optimization_deconvolution} and \autoref{fig:optimization_stellar_recovery}.

\begin{figure}[h]
\centering
\includegraphics[width=0.9\columnwidth]{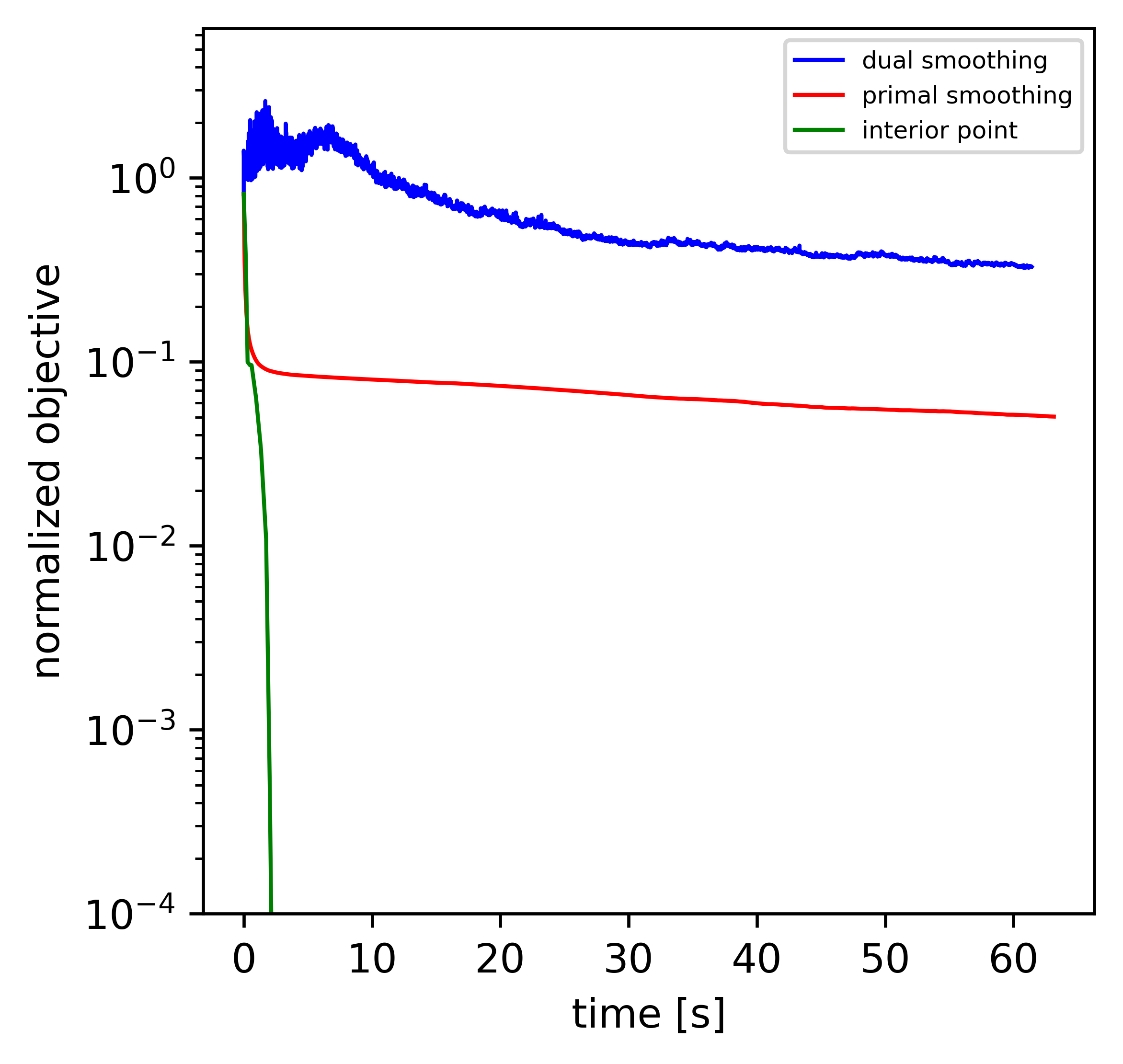}
\caption{Comparison of the computation time for the different optimization methods for the deconvolution problem (see \autoref{subsec:deconvolution}). The objective is normalized so that the minimum is at 0 and the initial value is at 1.}\label{fig:optimization_deconvolution}
\end{figure}

\begin{figure}[h]
\centering
\includegraphics[width=0.9\columnwidth]{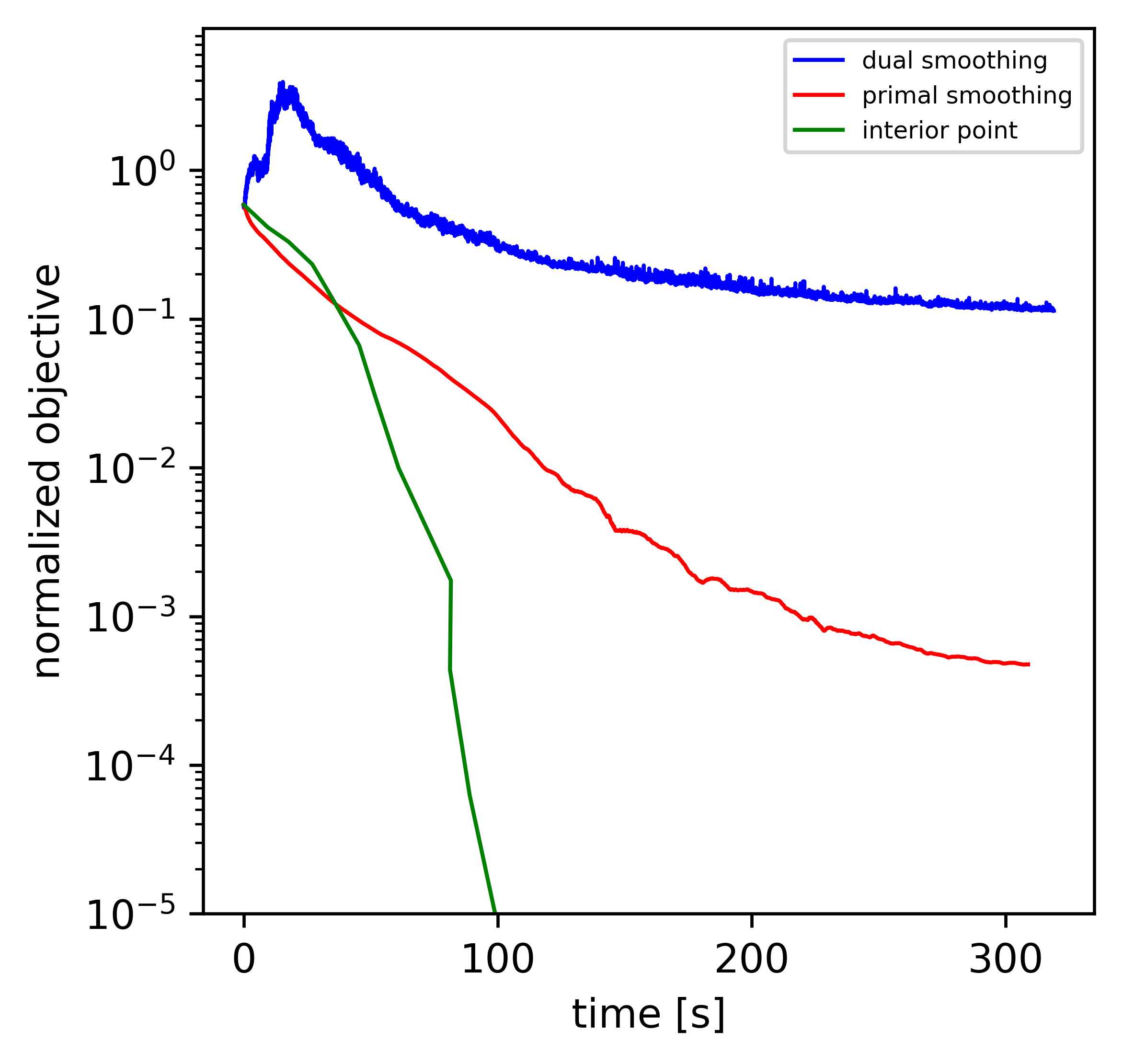}
\caption{Comparison of the computation time for the different optimization methods for the stellar recovery problem (see \autoref{subsec:stellar_recovery}). The objective is normalized so that the minimum is 0.}\label{fig:optimization_stellar_recovery}
\end{figure}

In both cases, the interior-point approach is able to achieve much higher accuracy. For the deconvolution problem, the first-order methods take considerably longer and are not able to achieve a high precision. For the stellar recovery problem, the primal smoothing method converges faster to a low-accuracy solution, but is not able to further improve.

Also, both the dual- and primal-smoothing approach have the additional disadvantage that they depend on the choice of a smoothing parameter $\mu$, where a larger value of $\mu$ corresponds to more smoothing but higher approximation error (see \autoref{rem:smoothing_parameter}). In \autoref{fig:smoothing_parameter}, we plotted the performance of the dual and primal smoothing approach on the stellar recovery problem for different choices of $\mu$. The trade-off between speed of convergence and achieveable accuracy is clearly visible. If the smoothing parameter is chosen too small, the first-order methods do not converge in a practical amount of time. This was in particular the case for the choice of smoothing parameter suggested by the bound \eqref{eq:nesterov_bound}. We tested many different choices, but did not find a configuration for which the performance of the first-order methods was comparable to the interior-point method. Furthermore, we also tested other solvers for the smoothed optimization problem, such as the L-BFGS-B method, but the performance was similar to the FGP method.

\begin{figure}[h]
\centering
\includegraphics[width=0.9\columnwidth]{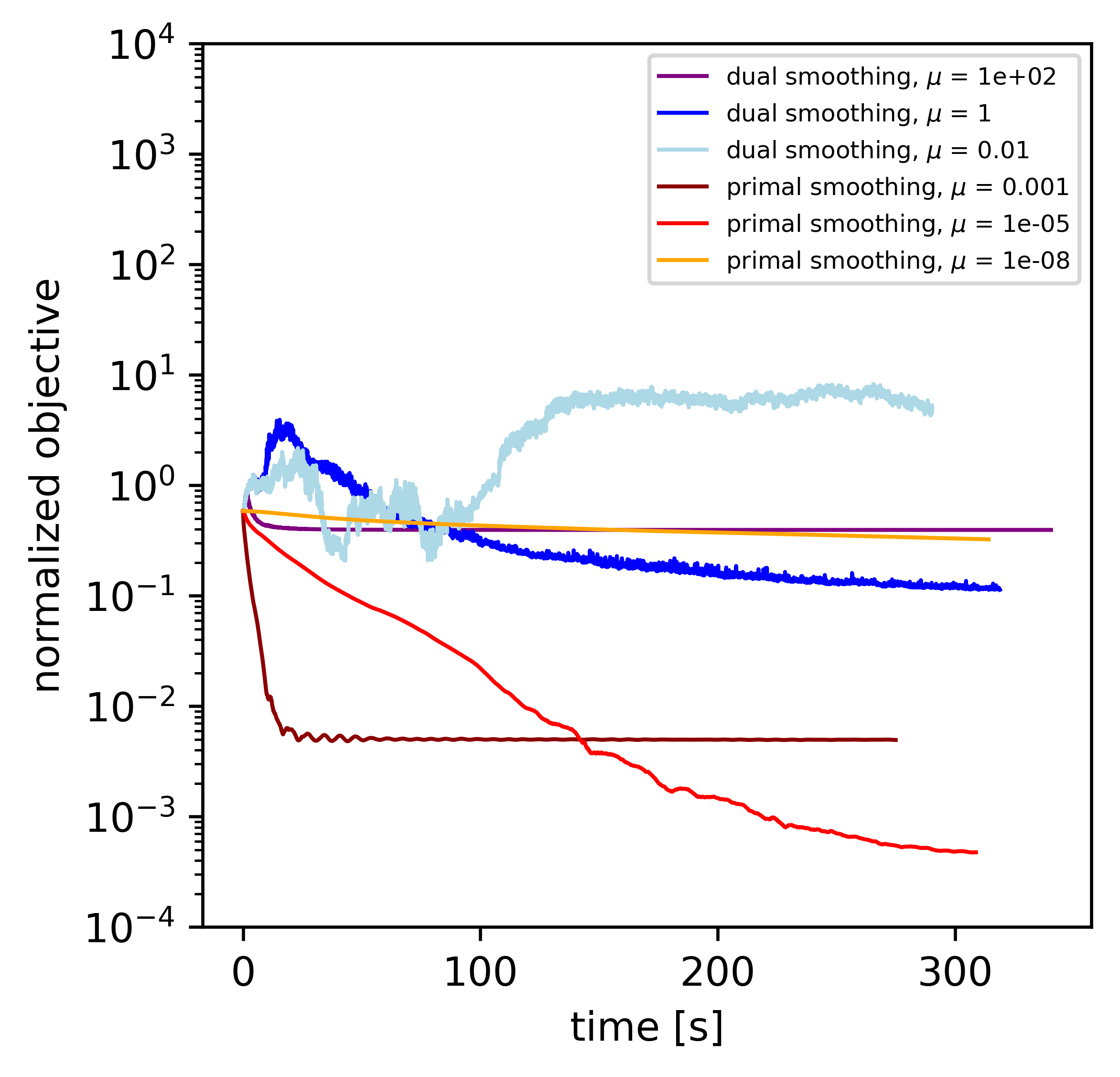}
\caption{The performance of the dual- and primal-smoothing approach for the stellar recovery problem, for various choices of smoothing parameter $\mu$. For each choice the plot shows the performance of 200 000 FGP iterations.}\label{fig:smoothing_parameter}
\end{figure}

Since it is necessary to achieve very high precision in the objective function of \eqref{eq:tv_ulog_discrete} to obtain the desired piecewise-constant solutions, the interior-point method should be the method of choice, since it has the additional advantage that it does not require hand-tuning of a smoothing parameter.

\section{Conclusion}\label{sec:conclusion}

In this work, we have developed a novel approach for blob detection in uncertain images that represents the uncertainty in the detected blobs by regions in scale space. These regions are obtained from the minimizer of a non-smooth optimization problem. Using similarities to CTV-denoising, we proposed three approaches for the numerical solution of the discretized problem. We also described how the scale space regions can be visualized on the image domain in an interpretable way.

The proposed method was illustrated on two numerical examples  -- one-dimensional deconvolution and integrated-light stellar population recovery -- where it yielded clear results that were consistent with the ground truth. We also evaluated the performance of the different optimization methods and observed that the interior-point method outperformed the other two approaches, assumedly because it is more robust against the ill-conditioning of the problem.

Our proposed method is flexible since it only requires access to a tube in which the scale-space representation of the uncertain image lies with high probability. Such a tube can be computed for many applications, for example in the important special case of Bayesian imaging.

Finally, the proposed methods are not specific to astronomical applications, although they were originally developed in that context. The methodology could be applied in any setting were blob detection in uncertain signals is relevant, for example in medical and geophysical imaging.

\backmatter

\bmhead{Acknowledgments}

FP and OS were supported by the Austrian Science Fund (FWF),
with SFB F68 ``Tomography Across the Scales'', project F6807-N36 (Tomography with Uncertainties). The financial support by the Austrian Federal Ministry for Digital and Economic
Affairs, the National Foundation for Research, Technology and Development and the Christian Doppler
Research Association is gratefully acknowledged. 
The authors are grateful to Prashin Jethwa and Glenn van de Ven for the productive cooperation within the SFB. The authors would also like to thank Noemi Naujoks for feedback on the manuscript.

\begin{appendices}

\section{Estimation of scale-space tubes from MCMC samples}\label{app:tubes_from_mcmc}

In this appendix, we describe how, given an observation $\Y = \y$, we can estimate a credible scale-space tube $[\Ulow, \Uupp]$ using MCMC samples $\f^{(1)}, \ldots, \f^{(S)}$ from the posterior distribution $\PP_{\F | \Y}(\cdot \given \y)$. Recall that MCMC is a class of methods that produce approximate samples from a probability distribution using its density function \cite{BroGelJonMen11}. Given MCMC samples, the standard way to estimate the posterior probability that $\F$ is contained in a given set $A \subset \R^{N_1 \times N_2}$ is \cite{MeyTwe09}
\begin{align}
\PP_{\F | \Y}(A \given \y) \approx \frac{1}{S} \sum_{s=1}^S \indicator_{A}(\f^{(s)}), \label{eq:mcmc_approx}
\end{align}
where the indicator function $\indicator_A$ is defined as
\begin{align*}
\indicator_A(\f) := \begin{cases}
1, & \text{if } \f \in A, \\
0, & \text{otherwise.}
\end{cases}
\end{align*}
For our purposes, we want to find the minimal-volume tube $[\Ulow, \Uupp]$ such that
\begin{align}
\PP_{\bm U | \Y}([\Ulow, \Uupp] \given \y) \geq 1 - \alpha. \label{eq:bayesian_credibility_condition}
\end{align}
Recall that for a set $A \subset \R^{N_1 \times N_2 \times K}$, $\PP_{\bm U | \Y}(A \given \y)$ is given by the pushforward
\begin{align}
\PP_{\U | \Y}(A \given \y) = \PP_{\F | \Y}(\bm \Phi^{-1}(A) \given \y), \label{eq:pushforward}
\end{align} 
where $\bm \Phi$ denotes the operator that maps an image $\f$ to its discrete scale-space representation (cf. \autoref{subsec:discrete_scale_space}). Inserting \eqref{eq:pushforward} in \eqref{eq:bayesian_credibility_condition} and using the approximation \eqref{eq:mcmc_approx}, we obtain
\begin{align*}
1 - \alpha & \leq \PP_{\F | \Y}(\bm \Phi^{-1}([\Ulow, \Uupp]) \given \y)\\
& \approx \frac{1}{S} \sum_{s=1}^S \indicator_{\bm \Phi^{-1}([\Ulow, \Uupp])}(\f^{(s)}) \\
& = \frac{1}{S} \abs{ \Set{s \in [S]}{\bm \Phi(\f^{(s)}) \in [\Ulow, \Uupp]} },
\end{align*}
where for a set $A$, $\abs{A}$ denotes its cardinality. This leads to a condition on $[\Ulow, \Uupp]$ in terms of samples:
\begin{align}
\abs{ \Set{s \in [S]}{\bm \Phi(\f^{(s)}) \in [\Ulow, \Uupp]} } \geq (1 - \alpha) \cdot S. \label{eq:approximate_credibility}
\end{align}
Given samples $\f^{(1)}, \ldots, \f^{(S)}$ it thus remains to find the smallest-volume tube $[\Ulow, \Uupp]$ that satisfies \eqref{eq:approximate_credibility}. However, as was already discussed in \cite[section 5.3]{ParJetBoeAlfSch23a}, solving this problem exactly is computationally infeasible even for moderately sized images. Instead, one has to use heuristic approaches that determine a small but not minimal tube that satisfies \eqref{eq:approximate_credibility}.

Our method is based on the idea of sorting the samples in descending order according to their probability, and then using bisection to find the smallest-volume tube that satisfies \eqref{eq:approximate_credibility} among the increasing sequence of tubes spanned by the ordered samples. In detail: Let 
$p_{\F | \Y}$ denote the density function of $\PP_{\F | \Y}$. Using a sorting algorithm, we find a reordering $(i_s)_{s=1}^S$ such that
\begin{align}
p_{\bm F | \Y}(\bm f^{(i_1)} \given \y) \geq \ldots \geq p_{\bm F | \Y}(\bm f^{(i_S)} \given \y). \label{eq:ordered_samples}
\end{align}
Then, we compute for each sample $\f^{(i)}$ its scale-space representation $\bm u^{(i)} = \bm \Phi(\f^{(i)})$.
For each $s \in [S]$, the discrete scale-space representations $\bm u^{(i_1)}, \ldots, \bm u^{(i_s)}$ span a tube $[\Ulows, \Uupps]$ given by
\begin{equation}
\begin{aligned}
\Ulows_{i,j,k} &:= \min \Set{u_{ijk}^{(i_r)}}{r \in [s]}, \\
\Uupps_{i,j,k} &:= \max \Set{u_{ijk}^{(i_r)}}{r \in [s]}.
\end{aligned}\label{eq:llowslupps}
\end{equation}
(The tube $[\Ulows, \Uupps]$ is simply the smallest-volume tube that contains $\bm u^{(i_1)}, \ldots, \bm u^{(i_s)}$). Since these tubes are monotonically increasing, the desired tube can be found very efficiently using bisection. The detailed pseudocode for this method is given in \autoref{alg:credible_tubes}.

Note that this algorithm is different from the method used in our previous work \cite[section 5.3]{ParJetBoeAlfSch23a}, which in particular did not use evaluations of the posterior density function. In our numerical experiments, the new method performed better, in the sense that it yielded tubes that contained the same number of samples but had lower volume.

\begin{algorithm}
\caption{Credible scale-space tubes from samples}\label{alg:credible_tubes}
\begin{algorithmic}[1]
\Require{Samples $\f^{(1)}, \ldots, \f^{(S)}$ from $p_{\F|\Y}(\cdot \given \y)$, a credibility parameter $\alpha \in (0, 1)$, and a maximum number of bisection steps $M \in \N$.}
\State Find a reordering $i_1,\ldots,i_S$ such that \eqref{eq:ordered_samples} holds;
\State Compute $\u^{(i)} = \bm \Phi \f^{(i)}$ for all $i \in [S]$;
\State $S_\alpha = \lceil (1 - \alpha) S \rceil$;
\State Set $T = [\bm u^{\mathrm{low}, S_\alpha}, \bm u^{\mathrm{upp}, S_\alpha}]$ (cf. \autoref{eq:llowslupps});
\State $K = \abs{ \Set{s \in [S]}{\bm u^{(i_s)} \in T} }$;
\If{$K = S_\alpha$}
\Return $T$
\Else
\State Set $K_\text{min} = 1$ and $K_\text{max} = S_\alpha$;
\For{$m = 1,\ldots, M$}
\If{$K > S_\alpha$}
\State $K_\text{max} = K$;
\State $K = \frac{1}{2}(K_\text{min} + K)$;
\ElsIf{$K < S_\alpha$}
\State $K_\text{min} = K$;
\State $K = \frac{1}{2}(K + K_\text{max})$;
\Else
\State break;
\EndIf
\State Set $T = [\u^{\mathrm{low}, K}, \u^{\mathrm{upp}, K}]$;
\State $K = \abs{\Set{s \in [S]}{\bm u^{(i_s)} \in T}}$;
\EndFor
\State return $T$;
\EndIf
\end{algorithmic}
\end{algorithm}

\end{appendices}

\bibliography{ParKirSch23_preprint}% common bib file

\end{document}